\documentclass{amsart}
\usepackage{amsthm}
\usepackage{pinlabel}
\usepackage{enumitem}
\usepackage{color}
\usepackage{algorithm}
\usepackage{subcaption}
\usepackage{graphicx}
\usepackage{multicol}
\newlist{steps}{enumerate}{1}
\setlist[steps, 1]{label = Step \arabic*:}
\newtheorem{theorem}{\bf Theorem}[section]
\newtheorem{lemma}[theorem]{\bf Lemma}
\newtheorem{definition}[theorem]{\bf Definition}
\newtheorem{corollary}[theorem]{\bf Corollary}
\newtheorem{proposition}[theorem]{\bf Proposition}
\newtheorem{remark}[theorem]{\bf Remark}
\newtheorem{example}[theorem]{\bf Example}
\newcommand{\rme}{\mathrm{e}}
\newcommand{\rmi}{\mathrm{i}}
\newcommand{\rmd}{\mathrm{d}}

\newcommand{\defeq}{\mathrel{\mathop:}=}

\begin{document}

\title{Totally tangential $\mathbb{C}$-links and electromagnetic knots}

\author{Benjamin \textsc{Bode}}
\address{Instituto de Ciencias Matemáticas (ICMAT), Consejo Superior de Investigaciones Científicas (CSIC), Campus Cantoblanco UAM, C/ Nicolás Cabrera, 13-15, 28049 Madrid, Spain}
\address{Departamento de Matemática Aplicada a la Ingeniería Industrial, ETSIDI, Universidad Politécnica de Madrid, Rda. de Valencia 3, 28012 Madrid, Spain}
\email{benjamin.bode@upm.es}


\begin{abstract}
The set of real-analytic Legendrian links with respect to the standard contact structure on the 3-sphere $S^3$ corresponds both to the set of totally tangential $\mathbb{C}$-links as defined by Rudolph and to the set of stable knotted field lines in Bateman electromagnetic fields of Hopf type. It is known that every isotopy class has a real-analytic Legendrian representative, so that every link type $L$ admits a holomorphic function $G:\mathbb{C}^2\to\mathbb{C}$ whose zeros intersect $S^3$ tangentially in $L$ and there is a Bateman electromagnetic field $\mathbf{F}$ with closed field lines in the shape of $L$. However, so far the family of torus links are the only examples where explicit expressions of $G$ and $\mathbf{F}$ have been found. In this paper, we present an algorithm that finds for every given link type $L$ a real-analytic Legendrian representative, parametrised in terms of trigonometric polynomials. We then prove that (good candidates for) examples of $G$ and $\mathbf{F}$ can be obtained by solving a system of linear equations, which is homogeneous in the case of $G$ and inhomogeneous in the case of $\mathbf{F}$. We also use the real-analytic Legendrian parametrisations to study the dynamics of knots in Bateman electromagnetic fields of Hopf type. In particular, we show that no compact subset of $\mathbb{R}^3$ can contain an electromagnetic knot indefinitely.
\end{abstract}

\maketitle

\section{Introduction}\label{sec:intro}

Consider the standard contact structure $\xi$ on $S^3\subset\mathbb{C}^2$, given by the kernel of the global contact form
\begin{equation}\label{eq:contact}
-y_1\rmd x_1+x_1\rmd y_1-y_2\rmd x_2+x_2 \rmd y_2,
\end{equation}
where we use $z_j=x_j+\rmi y_j$, $=1,2$, as coordinates on $\mathbb{C}^2$. It is known that every link in $S^3$ is ambient isotopic to a real-analytic Legendrian link with respect to $\xi$ \cite{rudolphtt}, that is, the link should be a parametrised by real-analytic functions in $\mathbb{C}^2$ and its tangent bundle should be contained in $\xi$. We will review an argument for this existence statement at a later point. However, finding such an explicit representative, say as a parametric curve, for a given link type is not straightforward. We can approximate any Legendrian link by a set of real-analytic curves, but the result might no longer be Legendrian. Likewise, we can approximate any real-analytic link by a Legendrian link (in the $C^0$-norm), but the result might no longer be real-analytic. In this paper, we present an algorithm that finds for any given link type an ambient isotopic set of parametric curves that are Legendrian and given in terms of trigonometric polynomials. In particular, they are real-analytic.

The challenge of finding explicit real-analytic Legendrian representatives is an important step in the construction of examples of totally tangential $\mathbb{C}$-links and electromagnetic knots. Both are known to exist for any link type, but in both cases the only known explicit examples are torus links. We show that, once an explicit parametrisation of the link has been found via our algorithm, good candidates for the desired functions can be obtained by solving a system of linear equations.

We will discuss some numerical difficulties with this approach, which are the reason that at the moment we do not have new explicit examples.

Even without explicit expressions for the electromagnetic fields, the explicit real-analytic Legendrian knots found via our algorithm allow us to study the dynamics of electromagnetic knots and prove general results about their time evolution.

\subsection{Electromagnetic knots}
Topological properties of electromagnetic fields have garnered much attention over the last years. Among the most striking examples is Ra{\~n}ada's Hopfion, a finite-energy solution to Maxwell's equations in vacuum (first mentioned by Synge \cite{synge}, while Ra{\~n}ada found its topological properties \cite{ranada}) that has the property that at any moment in time all field lines of the corresponding electric and magnetic field are closed loops and any pair of them forms a Hopf link. It was speculated that this topological stability could explain the longevity of the unexplained phenomenon called ball-lightning \cite{ranadaball, ranadaball2}, where a bright ball of lightning can be seen for various seconds as opposed to an average duration of a magnitude of microseconds for a lightning strike. Arguments opposing this hypothesis have been published in \cite{noball}.

Ra{\~n}ada's Hopfion raises the question which other knots and links can arise in electromagnetic fields. In \cite{kedia} Kedia, Bialynicki-Birula, Peralta-Salas and Irvine used a construction by Bateman \cite{bateman} to find explicit solutions to Maxwell's equations in vacuum whose electric and magnetic part have a closed field line in the shape of any given torus knot for all time. Ra{\~n}ada's Hopfion can be seen as a special case of this family.

An electromagnetic field consists of two time-dependent vector fields $\mathbf{E}_t$ and $\mathbf{B}_t$, which together satisfy Maxwell's equations. The two fields can be combined into the \textit{Riemann-Silberstein vector} $\mathbf{F}=\mathbf{E}_t+\rmi \mathbf{B}_t$, which maps $\mathbb{R}^{3+1}$ to $\mathbb{C}^3$ using the electric field as its real part and the magnetic field as its imaginary part. A good review on the Riemann-Silberstein vector can be found in \cite{RS}.

Bateman's construction defines an electromagnetic field 
\begin{equation}
\mathbf{F}=h(\alpha,\beta)\nabla\alpha\times\nabla\beta
\end{equation} 
for any choice of $(\alpha,\beta):\mathbb{R}^{3+1}\to\mathbb{C}^2$ with 
\begin{equation}
\label{eq:1}
\nabla \alpha\times\nabla\beta=\rmi (\partial_t\alpha\nabla\beta-\partial_t\beta\nabla\alpha),
\end{equation}
and any function $h:\mathbb{C}^2\to\mathbb{C}$ that is holomorphic on a neighbourhood of the image of $(\alpha,\beta)$. Here the symbol $\nabla$ refers to the gradient with respect to the three spatial variables $x$, $y$ and $z$. The family of Bateman fields in \cite{kedia} all use the same choice of $(\alpha,\beta)$, namely
\begin{align}\label{eq:alphabeta}
\alpha&=\frac{x^2+y^2+z^2-t^2-1+2\rmi z}{x^2+y^2+z^2-(t-\rmi)^2},\nonumber\\
\beta&=\frac{2(x-\rmi y)}{x^2+y^2+z^2-(t-\rmi)^2},
\end{align} 
whose image is the 3-sphere $S^3$ of unit radius in $\mathbb{C}^2$. The pushforward of the Poynting vector field of any Bateman field by $(\alpha,\beta)|_{t=0}$ with this choice of $(\alpha,\beta)$ is tangent to the fibers of the Hopf fibration. We say that such a Bateman field is \textit{of Hopf type}.

The time evolution of the field lines of a Bateman field of Hopf type is then given by pushforwards by a 1-parameter family of diffeomorphisms $\Phi_t:\mathbb{R}^3\to\mathbb{R}^3$, where $\Phi_0$ is the identity \cite{bode}. To be precise, if $\mathbf{E}_t$ and $\mathbf{B}_t$ denote the electric and magnetic part of a Bateman field of Hopf type at time $t$, then
\begin{align}
(\Phi_{t=t_*})_*\mathbf{E}_{t=0}=&\frac{(1+x^2+y^2+(t_*-z)^2)^2}{t_*^4-2t_*^2(x^2+y^2+z^2-1)+(x^2+y^2+z^2+1)^2}\mathbf{E}_{t=t_*}\nonumber\\
(\Phi_{t=t_*})_*\mathbf{B}_{t=0}=&\frac{(1+x^2+y^2+(t_*-z)^2)^2}{t_*^4-2t_*^2(x^2+y^2+z^2-1)+(x^2+y^2+z^2+1)^2}\mathbf{B}_{t=t_*}.
\end{align}
It follows that the topology of the field lines is preserved for all time. If $L$ is a set of closed field lines (i.e., a set of periodic orbits) of $\mathbf{E}_0$, then $\Phi_t(L)$ is a set of closed field lines of $\mathbf{E}_t$ for all $t\in\mathbb{R}$ of the same link type.

Alternatively, the time evolution of a Bateman field of Hopf type can be understood as follows. There is a pair of vector fields $\mathbf{E}$ and $\mathbf{B}$ on the unit 3-sphere, which do not depend on time, and a family of projection maps $\varphi_t:S^3\to\mathbb{R}^3\cup\{\infty\}$ such that $\mathbf{E}_t=(\varphi_t)_*(\mathbf{E})$ and $\mathbf{B}_t=(\varphi_t)_*(\mathbf{B})$ up to a known overall real factor \cite{bode}. In other words, the time-dependent electric and magnetic fields are multiples of pushforwards of time-independent vector fields on $S^3$ by time-dependent projection maps. The relation between the two perspectives is that $\Phi_t=\varphi_t\circ\varphi_0^{-1}$ and $\varphi_{t_*}=\left((\alpha,\beta)|_{t=t_*}\right)^{-1}$.

We showed in \cite{bode} that any link type can be realised as a set of closed field lines of the electric and magnetic part of a Bateman field of Hopf type. This is based on the observation that the electric and magnetic part of a Bateman field of Hopf type are Legendrian, that is to say, the corresponding fields $\mathbf{E}$ and $\mathbf{B}$ on $S^3$ are everywhere tangent to the standard contact structure on $S^3$, defined via Eq.~\eqref{eq:contact}. It follows that every periodic orbit $L$ of $\mathbf{E}$ or $\mathbf{B}$ must be tangent to the standard contact structure on $S^3$, which means that by definition $L$ is a Legendrian knot.

All functions and fields considered are real-analytic. It was shown in \cite{bode} that for a given Legendrian link $L$ on $S^3$ there is a corresponding Bateman field with a set of closed field lines given by $\varphi_t(L)$ if and only if $L$ is real-analytic. The result that every link type arises as a topologically stable set of closed field lines then follows, since it is known that every link type can be represented by a real-analytic Legendrian link. However, the proof in \cite{bode} is not constructive and currently the family of torus knots remain the only examples where the corresponding electromagnetic field is known. The choice of $h(z_1,z_2)=z_1^{p-1}z_2^{q-1}$ results in a Bateman field of Hopf type for a $(p,q)$-torus link \cite{kedia}.

\subsection{Totally tangential $\mathbb{C}$-links}

We denote by $\rho:\mathbb{C}^2\to\mathbb{R}$ the function $\rho(z_1,z_2)=|z_1|^2+|z_2|^2$ and we write $\text{Reg}(V_G)$ for the set of regular points of $V_G=\{(z_1,z_2):G(z_1,z_2)=0\}$, for a holomorphic function $G:\mathbb{C}^2\to\mathbb{C}$.

\begin{definition}\label{def:tt}
A link $L$ in $S^3$ is called a totally tangential $\mathbb{C}$-link if there exists a neighbourhood $U$ of the 4-ball $D^4\subset\mathbb{C}^2$ and a holomorphic function $G:U\to\mathbb{C}$ such that
\begin{itemize}
\item $L=V_G\cap S^3$.
\item $V_G\cap \text{Int}(D^4)=\emptyset$, where $\text{Int}(D^4)$ denotes the open 4-ball.
\item $V_G\cap S^3$ is a non-degenerate critical manifold of index 1 of $\rho|_{\text{Reg}(V_G)}$. 
\end{itemize}
\end{definition}

It follows from the definition that the intersection $V_F\cap S^3$ is tangential at every point. The family of totally tangential $\mathbb{C}$-link can thus be seen as the natural counterpart to the family of transverse $\mathbb{C}$-links, which are defined as transverse intersections $V_G\cap S^3$ \cite{rudolph83, rudolph84, boileau}. Occasionally, we simply say that $L$ is totally tangential if it is a totally tangential $\mathbb{C}$-link.

\begin{theorem}[Rudolph \cite{rudolphtt}]
A link is totally tangential if and only if it is real-analytic and Legendrian with respect to the standard contact structure on $S^3$.
\end{theorem}
\begin{corollary}[Rudolph \cite{rudolphtt}]
Every link type has a representative that is totally tangential.
\end{corollary}

Therefore, a link $L$ in $\mathbb{R}^3$ is a set of closed field lines of an electric or magnetic part of a Bateman field of Hopf type at time $t$ if and only if $\varphi_t^{-1}(L)$ is totally tangential.

\begin{example}[Rudolph \cite{rudolphtt}]
Consider the complex polynomial $G:\mathbb{C}^2\to\mathbb{C}$, $G(z_1,z_2)=z_1^pz_2^q-1$ for $p,q\in\mathbb{N}$ with $\gcd(p,q)=1$. Then the zeros of $G$ intersect a 3-sphere of appropriate radius totally tangentially in the $(p,-q)$-torus knot.
\end{example}

The situation regarding totally tangential $\mathbb{C}$-links is thus very similar to the state of the art in knotted electromagnetic fields. It is known that every link type has a representative that is a totally tangential $\mathbb{C}$-link, but the only examples for which the corresponding holomorphic function is actually known are torus links.

\subsection{Main results}

In this article we present an algorithmic procedure that finds a real-analytic Legendrian representative for any given link type.

\begin{theorem}\label{thm:algo}
Algorithm 1 constructs for any given link $L$ a real-analytic Legendrian link, presented as a set of parametric curves in $S^3\subset\mathbb{R}^4$, that is ambient isotopic to $L$. In this parametrisation each coordinate is a quotient of two trigonometric polynomials.
\end{theorem}

In fact, the algorithm constructs trigonometric parametrisations of Legendrian links in $\mathbb{R}^3$ with respect to an appropriate contact structure. The fact that the parametrisation of the corresponding link in $S^3$ is given by quotients of trigonometric polynomials follows from properties of the projection maps between $S^3$ and $\mathbb{R}^3$.

We see this algorithm as a first step in the generation of explicit examples of electromagnetic knots and totally tangential $\mathbb{C}$-links, or rather, the functions $h$ and $G$ that define them, respectively.

\begin{theorem}\label{thm:2}
Let $L\subset S^3$ be an output of Algorithm 1. 
\begin{enumerate}[label={\roman*)}]
\item Then for every $n$ there is a matrix $A$ (whose size depends on $n$ and the parametrisation of $L$) such that any non-trivial solution $x$ of $Ax=0$ corresponds to the coefficients of a polynomial $G:\mathbb{C}^2\to\mathbb{C}$ of degree at most $2n$ and with $G|_L=0$. For sufficiently large $n$ a non-trivial solution always exists.

\item Furthermore, for every $n$ there is a vector $y$ (whose size depends on $n$ and on the parametrisation of $L$) such that every solution $x$ of $Ax=y$ corresponds to the coefficients of a polynomial $h:\mathbb{C}^2\to\mathbb{C}$ of degree at most $2n$ and with the property that $\mathbf{F}=h(\alpha,\beta)\nabla\alpha\times\nabla\beta$ is a Bateman field of Hopf type with $L$ as a set of closed field lines of its magnetic (or electric) part.
\end{enumerate}
\end{theorem}

\begin{remark}
\ 
\begin{itemize}
\item The polynomials $G$ in Theorem~\ref{thm:2} do not necessarily satisfy all the properties from Definition~\ref{def:tt}. They could have other zeros on $S^3$ besides $L$, so that we have in general $L\subseteq V_G\cap S^3$ instead of $L=V_G\cap S^3$. Likewise, we cannot rule out other zeros of $G$ in the inside of the unit 4-ball. It is not automatically true that the points on $L$ are regular points of $G$. But if that is the case, then the intersection is tangential along $L$. Therefore, if all points on $L$ are regular and the second property is satisfied (no zeros in $\text{Int}(D^4)$), then the third property in Definition~\ref{def:tt} holds. In particular, the intersection between $V_G$ and $S^3$ is tangential everywhere on $L$. The hope is that Algorithm 1 and Theorem~\ref{thm:2} can be used to obtain examples, for which we can check the properties in Definition~\ref{def:tt} individually.
\item It is known that every link type can be realised as a totally tangential $\mathbb{C}$-link by an entire function $G$ \cite{rudolphtt2}, so that $U$ in Definition~\ref{def:tt} is all of $\mathbb{C}^2$. It is not known if every link type can be obtained from a polynomial $G$.
\item While the equation $Ax=0$ in Theorem~\ref{thm:2} has non-trivial solutions if $n$ is sufficiently large, it is not known if the corresponding inhomogeneous equation $Ax=y$ has solutions for any value of $n$. In particular, it is not known if every link type arises as a closed field line in a Bateman field of Hopf type, if the defining function $h$ is required to be a polynomial.
\end{itemize}
\end{remark}

Theorem~\ref{thm:2} outlines a strategy to find good candidates $G$ for examples of functions that define totally tangential $\mathbb{C}$-links and examples of knotted electromagnetic fields. There are however several difficulties in practice. One is that the matrix $A$ in Theorem~\ref{thm:2} is very large, even for comparatively simple knots. Finding solutions should be possible numerically, but requires sufficient computing power (or a lot of patience). Also note that the property of a tangential intersection is a very subtle condition, which is not stable under small perturbations of the functions. A solution to the matrix equation should be thought of as an approximation to the actual desired function $G$, within numerical accuracy of the used computing device. There are some other numerical intricacies that we will explain at a later point. It is for these reasons that despite this machinery the search for examples of non-torus links remains a challenging open problem.

Even though ultimately Algorithm 1 has not led to new examples of knotted Bateman fields yet, we show that we can use the resulting curves to study the time evolution of knots in Bateman fields of knot type even without knowing an expression for the surrounding electromagnetic field. In particular, we show the following.

\begin{theorem}\label{thm:dynamics}
Let $\mathbf{F}$ be a Bateman field of Hopf type and let $L$ be a set of closed field lines of its magnetic (or electric) part at time $t=0$. Recall that there is a smooth family of diffeomorphisms $\Phi_t$ of $\mathbb{R}^3$ such that $\Phi_t(L)$ is a set of closed field lines at time $t$. Then for every compact set $K$ in $\mathbb{R}^3$ the set of values of $t$ for which $\Phi_t(L)\cap K\neq\emptyset$ is bounded. 
\end{theorem}

In other words, as the electromagnetic field evolves with time, the knot moves, growing indefinitely, so that no compact set can contain parts of it indefinitely.

The remainder of this paper is structured as follows. In Section~\ref{sec:algo} we outline Algorithm 1, which constructs real-analytic Legendrian links, and prove Theorem~\ref{thm:algo}. Section~\ref{sec:ex} details the individual steps for the figure-eight knot. We then discuss some variations of the algorithm in Section~\ref{sec:variations}. We prove Theorem~\ref{thm:2} in Section~\ref{sec:linear} and discuss the dynamics of electromagnetic knots as well as a proof of Theorem~\ref{thm:dynamics} in Section~\ref{sec:dynamics}.

\section{Finding real-analytic Legendrian representatives}\label{sec:algo}

Before we describe Algorithm 1, which constructs real-analytic Legendrian representatives of any give link type in $S^3$, we briefly review a typical argument why such a representative always exists. It can be found in \cite{rudolphtt}, where it is attributed to Eliashberg. Suppose that $L$ is a Legendrian link in $S^3$ with the standard contact structure. Then there is an arbitrarily close real-analytic approximation $L'$ of $L$. As pointed out before, this new link $L'$ is not necessarily a Legendrian link with respect to the standard contact structure. But, since $L'$ is ambient isotopic to $L$, it is Legendrian with respect to some other contact structure on $S^3$ (which is contactomorphic to the standard one, but different on the nose). Actually, every ambient isotopy of $S^3$ that carries $L$ to $L'$ results in a contact structure $\xi$ of $S^3$ that is equivalent to the standard one and with $L'$ Legendrian with respect to $\xi$. By choosing the approximation sufficiently close, we can guarantee that among this family of contact structures there exists one that is equivalent to the standard one via a real-analytic diffeomorphism $\varphi$. Thus the image $\varphi(L')$ is a real-analytic Legendrian link with respect to the standard contact structure on $S^3$ and since $\varphi$ is an orientation-preserving diffeomorphism, it has the same link type as $L$ and $L'$.

This argument is not particularly helpful when we try to write down an actual analytic expression for the desired space curve. Finding a real-analytic approximation to a given link is easily done, but finding the real-analytic diffeomorphism that maps the deformed contact structure back to the standard contact structure is highly non-trivial. 

The goal of this section is to find real-analytic Legendrian representatives with respect to the standard contact structure on the 3-sphere. However, there exist various real-analytic projection maps from (subsets of) $S^3$ to $\mathbb{R}^3$, which allow us to first solve the problem on $\mathbb{R}^3$ with the contact structure given by the pushforward of the standard contact structure on $S^3$ by these projection maps, and then map the corresponding real-analytic Legendrian links back to $S^3$. In fact, we will construct Legendrian links in $\mathbb{R}^3$ that are parametrised by trigonometric polynomials (i.e., real finite Fourier series), which will be useful in our attempts to construct the holomorphic functions defining totally tangential $\mathbb{C}$-links and Bateman eletromagnetic fields of Hopf type.

Let $H_+=\{(z_1,z_2)|\text{Re}(z_1)>0\}$ be the right open half-space of $\mathbb{C}^2$, $S^3_+:=S^3\cap H_+$ the right open half-sphere and $T$ the tangent 3-space to $S^3$ at $(1,0)$. Throughout this section let $L$ be a knot in $S^3$ parametrised by $2\pi$-periodic real-analytic functions $(z_1(t),z_2(t)) = (x_1(t)+iy_1(t),x_2(t)+iy_2(t))$. We state our results for knots, but since a link is a totally tangential $\mathbb{C}$-link if and only if all of its components are totally tangential $\mathbb{C}$-links, all results remain true for links of any number of components.

\begin{lemma}\label{lemma}
The knot $L$ is totally tangential if and only if $\overline{z_1}(t)z_1'(t)+\overline{z_2}(t)z_2'(t)=0$ for all $t\in[0,2\pi]$.
\end{lemma}
\begin{proof}
The real part of $\overline{z_1}(t)z_1'(t)+\overline{z_2}(t)z_2'(t)$ is equal to $\tfrac{\partial (|z_1(t)|^2+|z_2(t)|^2)}{\partial t}$ and thus its vanishing is equivalent to $L$ lying on a 3-sphere. The imaginary part of $\overline{z_1}(t)z_1'(t)+\overline{z_2}(t)z_2'(t)$ is equal to $x_1'(t)y_1(t)-x_1(t)y_1'(t)+x_2'(t)y_2(t)-x_2(t)y_2'(t)$, which vanishes if and only if the tangent vector to $L$ at each point lies in the complex line tangent to $S^3$ at that point, or, equivalently, if $L$ is tangent to the standard contact structure at every point.
\end{proof}

Now assume that $L$ lies in $S^3_+$ and consider $L'\subset T$, the image of $L$ under radial projection from the origin of $\mathbb{C}^2$ into $T$. It is parametrised, in real coordinates, by $(1,y_1(t)/x_1(t),x_2(t)/x_1(t),y_2(t)/x_1(t))$.  We write $X(t)=x_2(t)/x_1(t)$, $Y(t)=y_2(t)/x_1(t)$, $Z(t)=y_1(t)/x_1(t)$.

\begin{proposition}\label{thm:yuv}
A knot $L$ in $S^3_+$ is a totally tangential $\mathbb{C}$-link if and only if $(X(t),Y(t),Z(t))$ satisfies $Z'+XY'-YX'=0$.
\end{proposition}
\begin{proof}
The proposition can be verified by direct calculation. We find
\begin{equation}
Z'+XY'-YX'=\frac{x_1y_1'-x_1'y_1}{x_1^2}+\frac{x_2 (x_1y_2'-y_2x_1')}{x_1^3}-\frac{y_2 (x_1x_2'-x_2x_1')}{x_1^3},
\end{equation}
which vanishes if and only if 
\begin{align}\label{eq:calc}
&(x_1y_1'-x_1'y_1)x_1+x_2 (x_1y_2'-y_2x_1')-y_2 (x_1x_2'-x_2x_1')\nonumber\\
=&(x_1y_1'-x_1'y_1)x_1+x_1x_2 y_2'-x_1y_2 x_2'
\end{align}
vanishes, since $x_1>0$. But the right hand side of Eq.~\eqref{eq:calc} is exactly $x_1\text{Im}(\overline{z_1}(t)z_1'(t)+\overline{z_2}(t)z_2'(t))$, whose vanishing (under the assumption that $L$ lies in $S^3$) is by Lemma \ref{lemma} equivalent to $L$ being totally tangential.
\end{proof}

Proposition \ref{thm:yuv} provides us with a tool to construct totally tangential $\mathbb{C}$-links. We need to find real analytic $2\pi$-periodic functions $X(t)$ and $Y(t)$, so that $Z(t)=\int_{0}^tY(s)X'(s)-X(s)Y'(s)ds$ is also $2\pi$-periodic and such that $(X(t),Y(t),Z(t))$ is a parametrisation of $L$ in $\mathbb{R}^3$. We can then identify $\mathbb{R}^3$ with $T$ and project $L$ radially to $S^3_+$. Since the projection map is real-analytic, the result is then by Proposition \ref{thm:yuv} a totally tangential $\mathbb{C}$-link. Below we describe a procedure that finds the desired functions $X(t)$ and $Y(t)$ as above for any given link type.

Lemma~\ref{lemma} and Proposition~\ref{thm:yuv} also follow directly from the observation that the radial projection map is a (real-analytic) diffeomorphism mapping the standard contact structure on $S^3$ to  the contact structure on $\mathbb{R}^3$ defined via the form $\rmd Z+X\rmd Y-Y\rmd X$, see \cite[Examples 2.1.3 and 2.1.10]{geiges} for the formula for the lower hemisphere and an overall rotation in $\mathbb{R}^3$.


We now describe the algorithm that finds a real-analytic Legendrian representative for any given link type. Following the observation above we want to find a real-analytic parametrisation $(X(t),Y(t),Z(t))$ of $L$ in $T\cong\mathbb{R}^3$ that satisfies $Z'+XY'-YX'=0$. The projection of this curve on $S^3$ is then the desired representative.

The idea of the construction is to start with a parametrisation $(X(t),Y(t))$ of a knot diagram $D$ and to define $Z(t)$ by integerating $YX'-XY'$. The resulting parametric curve $(X(t),Y(t),Z(t))$ is not necessarily a closed loop and the crossings potentially do not have the correct sign. (The crossings might also actually be intersection points). The following lemma helps us fix these issues.

\begin{lemma}\label{lem:circle}
Let $C$ be a circle of radius $r$ parametrised by $(X(t),Y(t))=(r\cos(t)+x,r\sin(t)+y)$ and let $Z(t)=\int_0^tY(s)X'(s)-X(s)Y'(s)\rmd s$. Then $Z(2\pi)=-2\pi r^2$ and $Z(-2\pi)=2\pi r^2$.
\end{lemma}
\begin{proof}
The lemma is proved by a direct calculation.
\end{proof}

Thus we can insert small circles into the knot diagram, which are used to ``wind up'' or ``wind down'' the $Z$-coordinate, so that we obtain a closed loop $(X(t),Y(t),Z(t))$ with the desired crossing signs. This procedure is described in more detail below. The algorithm can be outlined as follows.

\begin{algorithm}[H]
\caption{Construction of real-analytic Legendrian represetatives}
\begin{steps}
\item Find a $C^1$-parametrisation $(X(t),Y(t))$ of a knot diagram, for example using the algorithm in \cite{bodepoly, semiAK}.
\item Define $Z(t)=\int_{0}^tY(s)X'(s)-X(s)Y'(s)\rmd s$.
\item Check the sign of each crossing and insert circles into the knot diagram as required.
\item Insert an appropriate circle into the knot diagram if $Z$ is not $2\pi$-periodic.
\item Define $(\widetilde{X}(t),\widetilde{Y}(t))$ as the concatenation of arcs from $(X(t),Y(t))$ and the inserted circles.
\item Approximate the curve $(\widetilde{X}(t),\widetilde{Y}(t))$ given by the union of $(X(t),Y(t))$ and the circles by finite Fourier series $(X_{trig}(t),Y_{trig}(t))$.
\item Turn $(X_{trig}(t),Y_{trig}(t))$ into a \textit{balanced parametrisation}.
\item Define $Z_{trig}(t)=\int_{0}^tY_{trig}(s)X_{trig}'(s)-X_{trig}(s)Y_{trig}'(s)\rmd s$.
\item Compute $(x_1(t),y_1(t),x_2(t),y_2(t))=(\tfrac{1}{R},\tfrac{X_{trig}(t)}{R},\tfrac{Y_{trig}(t)}{R},\tfrac{Z_{trig}(t)}{R})$, where $R=\sqrt{1+X_{trig}^2(t)+Y_{trig}^2(t)+Z_{trig}^2(t)}$.
\end{steps}
\end{algorithm}

The individual steps are described in the following subsections, along with explanations of why this procedure produces real-analytic Legendrian representatives.

\subsection{Steps 1--5: Legendrian curves}

We start with a pair of $C^1$, $2\pi$-periodic functions $(X(t),Y(t))$ that parametrise a diagram $D$ of the knot $L$ such that $(X'(t),Y'(t))\neq(0,0)$ for all $t\in[0,2\pi]$. Such a parametrisation in terms of trigonometic functions can for example be found with the methods from \cite{bodepoly, semiAK}. In particular, the arguments from \cite{semiAK} guarantee that every crossing in the diagram only involves two arcs that intersect transversely. A crossing in the diagram thus corresponds to values $t'<t''\in[0,2\pi)$ with $X(t')=X(t'')$ and $Y(t')=Y(t'')$. Denote the values of $t$ at which there is a crossing by $t_j$, $j=1,2,\ldots,2c$, where $c$ is the number of crossings. Note that there are exactly $2c$ such values of $t\in[0,2\pi)$, since every crossing involves exactly two strands. Set $t_0=0$ and $t_{2c+1}=2\pi$. This concludes Step 1. Step 2 is simply the definition of $Z(t)\defeq\int_{0}^tY(s)X'(s)-X(s)Y'(s)\rmd s$, which guarantees that the resulting parametric curve satisfies the differential equation.

In Step 3 we have to check for each crossing of $D$, whether the function $Z$ induces the correct sign. Note that the sign of the crossing is determined by the fact that the overpassing arc has a greater $Z$-coordinate than the underpassing arc. Thus for every crossing we have to compare $Z(t_j)$ and $Z(t_{j'})$, where $(X(t_{j'}),Y(t_{j'})=(X(t_j),Y(t_j))$ and $t_j\neq t_{j'}$. We label the crossings that do not obtain the desired sign by $k=1,2,\ldots,N$ and write $j(k)$, $j'(k)$ for the indices with $t_{j(k)}<t_{j'(k)}$ and the crossing $k$ at $(X(t_{j(k)}),Y(t_{j(k)}))=(X(t_{j'(k)}),Y(t_{j'(k)}))$. For every $k\in\{1,2,\ldots,N\}$ we choose intermediate values $\tau_k\in(t_{j(k)-1},t_{j(k)})$ and $\tau_k'\in(t_{j(k)},t_{j(k)+1})$.

\begin{definition}\label{def:leftof}
Let $C$ be a circle with center $c$ that intersects a parametric curve $D=(X(t),Y(t))$ tangentially in a unique point $(X(t'),Y(t'))$ for some $t'$ with $(X'(t'),Y'(t'))\neq(0,0)$. Then we say that $C$ is on the left of $D$ if $(X'(t'),Y'(t'))\times(c-(X(t'),Y(t')))>0$, where $\times$ denotes the cross-product of 2-dimensional vectors. Otherwise, we say that $C$ is on the right of $D$. 
\end{definition}

For every $k\in\{1,2,\ldots,N\}$ we insert two circles $C_1(k)$, $C_2(k)$ of equal, small radii $r(k)=r_1(k)=r_2(k)$ into the diagram such that
\begin{itemize}
\item $C_1(k)$ intersects $D$ tangentially in a unique point, which is $(X(\tau_k),Y(\tau_k))$.
\item $C_2(k)$ intersects $D$ tangentially in a unique point, which is $(X(\tau_k'),Y(\tau_k'))$. 
\item If the desired crossing sign would be obtained by $Z(t_{j(k)})>Z(t_{j(k)'})$, then $C_1(k)$ is on the right of $D$ and $C_2(k)$ is on the left of $D$. Otherwise $C_1(k)$ is on the left of $D$ and $C_2(k)$ is on the right of $D$.
\item If the desired crossing sign would be obtained by $Z(t_{j(k)})>Z(t_{j(k)'})$, then $C_1(k)$ is oriented clockwise and $C_2(k)$ is oriented counter-clockwise. Otherwise, the orientations are reversed.
\end{itemize}

The first two properties ensure that the union of $(X(t),Y(t))$ and the inserted circles can still be parametrised as a $C^1$-loop in $\mathbb{R}^2$. The last two properties guarantee that by Lemma \ref{lem:circle} we can increase or decrease $Z(t_j(k))$ until it induces the correct sign for the crossing $k$ by traversing $C_1(k)$ a sufficient number of times in the direction of its orientation. By traversing $C_2(k)$ the same number of times we negate the effect of $C_1$, so that the signs of the other crossings are not affected. 

Let $m_k=\left\lfloor\tfrac{|Z(t_{j(k)})-Z(t_{j(k)'})|}{2\pi r(k)^2}\right\rfloor+1$. Then $m_k$ is a sufficient choice for the number of times that $C_1(k)$ and $C_2(k)$ each have to be traversed.

Likewise in Step 4 we have to check if $Z$ is $2\pi$-periodic. Choose a value $\tau_{N+1}\in(t_{2c},2\pi)$. If $Z(0)\neq Z(2\pi)$, we insert a small circle $C$ into the diagram such that
\begin{itemize}
\item $C$ intersects $D$ tangentially in a unique point $(X(\tau_{N+1},Y(\tau_{N+1}))$.
\item If $Z(0)>Z(2\pi)$, then $C$ is on the right of $D$ and oriented clockwise. Otherwise it is on the left of $D$ and oriented counter-clockwise.
\item The radius $r$ of $C$ is such that $\tfrac{|Z(0)-Z(2\pi)|}{2\pi r^2}$ is a natural number $m$.
\end{itemize}

We set $\tau_0'=0$ and $\tau_{N+1}'=2\pi$. Then the $\tau_k$s with $k=1,2,\ldots,N,N+1$ and $\tau_k'$s with $k=0,1,2,\ldots,N,N+1$, split the loop $(X(t),Y(t))$ into $2N+2$ arcs given by $A_k=(X(t),Y(t))$, $k=1,2,\ldots,N,N+1$, with $t$ going from $\tau_{k-1}'$ to $\tau_{k}$ and $B_k=(X(t),Y(t))$, $k=1,2,\ldots,N,N+1$, with $t$ going from $\tau_k$ to $\tau_k'$. Step 5 defines a new parametric curve $(\widetilde{X}(t),\widetilde{Y}(t))$ as the concatenation of the arcs $A_k$ and $B_k$ and the circles $C_1(k)$, $C_2(k)$ and $C$. In the partial groupoid of parametrised paths in $\mathbb{R}^2$, we can write the loop $(\widetilde{X}(t),\widetilde{Y}(t))$ as $(\prod_{k=1}^N A_kC_1(k)^{m_k}B_kC_2(k)^{m_k})A_{N+1}C^mB_{N+1}$. This would correspond to a loop, parametrised by two $2\pi(m+1+\sum_{k=1}^N2m_k)$-periodic functions. We prefer to view $(\widetilde{X}(t),\widetilde{Y}(t))$ as $2\pi$-periodic functions, so a rescaling of the $t$-coordinate is necessary. This concludes Step 5.

As usual we define $\widetilde{Z}(t)\defeq \int_{0}^t\widetilde{Y}(s)\widetilde{X}'(s)-\widetilde{X}(s)\widetilde{Y}'(s)\rmd s$.
\begin{lemma}\label{lem:iso}
The functions $(\widetilde{X}(t),\widetilde{Y}(t),\widetilde{Z}(t))$ parametrise a simple loop in $\mathbb{R}^3$ as $t$ goes from $0$ to $2\pi$. For every sufficiently small $\varepsilon>0$ any smooth $\varepsilon$-close approximation of $(\widetilde{X}(t),\widetilde{Y}(t),\widetilde{Z}(t))$ (with respect to the $C^1$-norm) is ambient isotopic to the desired knot $L$.
\end{lemma}
\begin{proof}
By inserting the circle $C$ we have ensured that $\widetilde{Z}(0)=\widetilde{Z}(2\pi)$, so that the parametric curve is a loop. An intersection point would either correspond to a crossing in the original curve $(X(t),Y(t))$ or project to a point on the inserted circles. By construction all crossings in the original curve have the desired signs and in particular there are no intersections. By Lemma \ref{lem:circle} the $Z$-coordinates of different points on the curve that project to the same point on an inserted circle differ by $2\pi r^2n$ for some non-zero integer $n$ and $r>0$. In particular, there are no intersection points that project to a point on an inserted circle. Therefore, the loop is simple.

We can apply a small, smooth isotopy to $(\widetilde{X}(t),\widetilde{Y}(t),\widetilde{Z}(t))$, so that the image of the loop is contained in a tubular neighbourhood of $(\widetilde{X}(t),\widetilde{Y}(t),\widetilde{Z}(t))$ and that changes the curve in a neighbourhood of the inserted circles as shown in Figure \ref{fig:spiral}. The small isotopy extends to an ambient isotopy, so that any smooth, sufficiently close approximation to the curve $(\widetilde{X}(t),\widetilde{Y}(t),\widetilde{Z}(t))$ is mapped to an $\varepsilon$-close approximation of the deformed curve in the $C^1$-norm. In particular, it produces the same knot diagram.

Note that all crossings that come from the inserted circles can be removed via Reidemeister moves of the first type. This results in the original diagram $D$ of $L$, so that the approximation of $(\widetilde{X}(t),\widetilde{Y}(t),\widetilde{Z}(t))$ is ambient isotopic to $L$.

\begin{figure}
\centering
\labellist
\Large
\pinlabel $m$ at 750 1150
\pinlabel $m$ at 750 200
\endlabellist
\includegraphics[height=5cm]{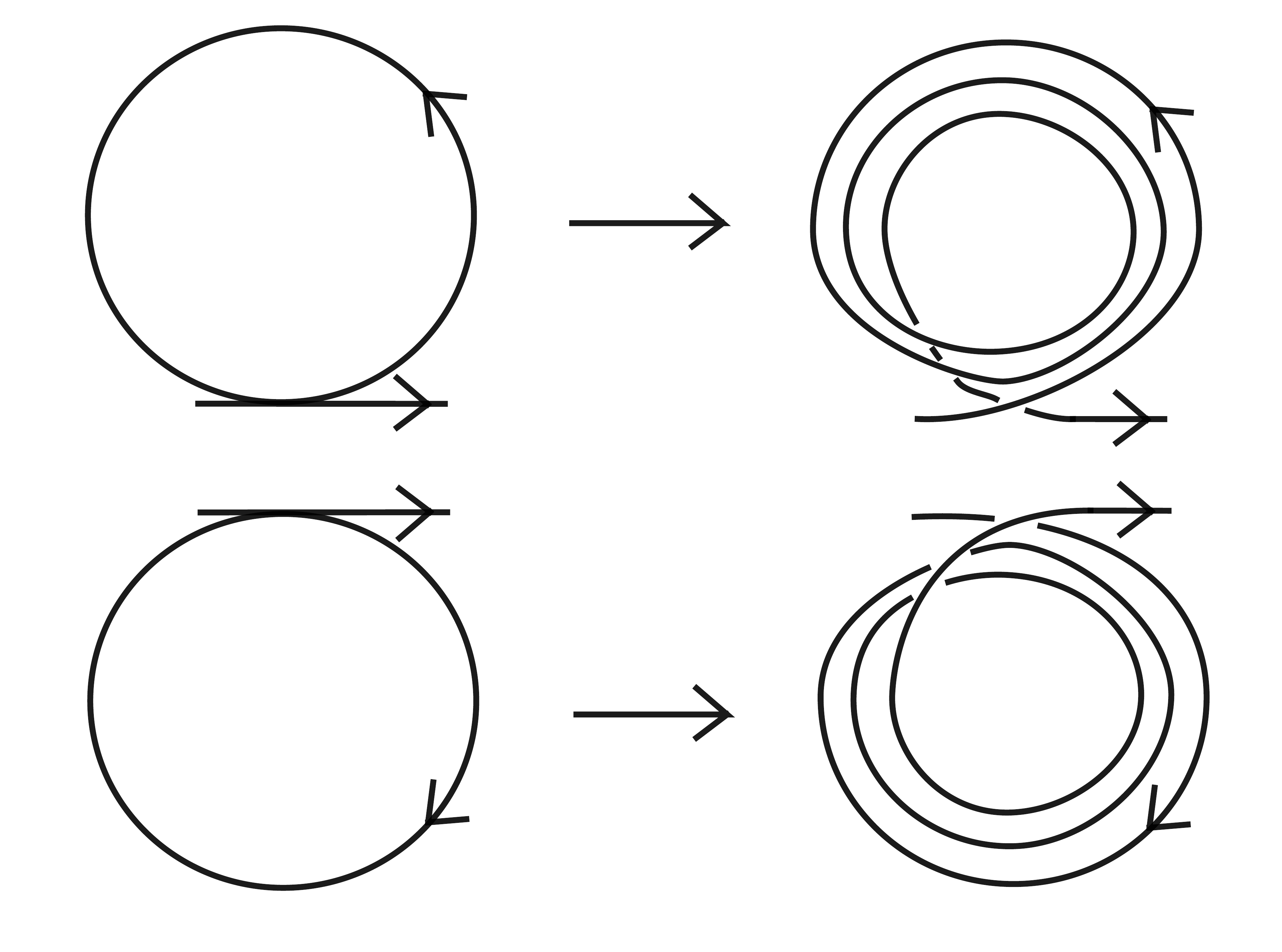}
\caption{An isotopy of the spirals. In this example $m=3$. \label{fig:spiral}}
\end{figure}
\end{proof}

\subsection{Steps 6--9: Fourier series}

In Step 6 we compute a finite Fourier series for $\widetilde{X}(t)+\rmi \widetilde{Y}(t)$, i.e., we calculate the first $n$ terms of its Fourier series for some sufficiently large $n$. The resulting finite Fourier series $X_{trig}(t)+\rmi Y_{trig}(t)$ is an $\varepsilon$-approximation of $\widetilde{X}(t)+\rmi \widetilde{Y}(t)$ in the $C^1$-norm, where $\varepsilon>0$ can be chosen arbitrarily small by increasing $n$. We do not expect a straighforward procedure that would calculate sufficient values for $n$. In practice, we have to start with some arbitrary value $n$, extract a Gauss code from its parametrisation and check (for example using pyknotid \cite{sandy}, SnapPy \cite{snappy} and/or other mathematical software) whether the knot is equivalent to $L$.  If it is not, we have to increase $n$ and repeat the process. 

We could define $Z_{trig}(t)$ as the integral over $Y_{trig}X_{trig}'-X_{trig}Y_{trig}'$ and obtain an $\varepsilon$-approximation to $\widetilde{Z}(t)$ with respect to the $C^1$-norm. However, $Z_{trig}$ is in general not $2\pi$-periodic.
\begin{lemma}\label{lem:balance}
Let $X_{trig}(t)=a_0+\sum_{j=1}^n(a_j\cos(jt)+b_j\sin(jt))$, $Y_{trig}(t)=c_0+\sum_{j=1}^n(c_j\cos(jt)+d_j\sin(jt))$. Then $Z_{trig}(2\pi)=2\pi\sum_{j=1}^n j(a_j d_j-b_j c_j)$.
\end{lemma} 
\begin{proof}
This is a direct calculation using identities of integrals of trigonometric functions.
\end{proof}

\begin{definition}
We say that $(X_{trig}(t),Y_{trig})$ as above is a balanced parametrisation if $\sum_{j=1}^n j(a_j d_j-b_j c_j)=0$.
\end{definition}
It follows immediately that $(X_{trig}(t),Y_{trig})$ is a balanced parametrisation if and only if $Z_{trig}(t)$ is $2\pi$-periodic. If this is not the case, we have to change the functions $(X_{trig}(t),Y_{trig}(t))$ in some controlled way to turn them into a balanced parametrisation.

We know that $Z_{trig}$ is an arbitrarily close approximation to $\widetilde{Z}$, which satisfies $\widetilde{Z}(0)=\widetilde{Z}(2\pi)=0$. Thus for any $\varepsilon<0$ we can choose $n$ sufficiently large so that $|Z_{trig}(2\pi)|=|2\pi\sum_{j=1}^n j(a_j d_j-b_j c_j)|=\varepsilon'$ for some $0\leq\varepsilon'<\varepsilon$. If $\varepsilon'=0$, then $Z_{trig}$ is $2\pi$-periodic and we are done.

If $\varepsilon'>0$ we have to change $(X_{trig}(t),Y_{trig}(t))$. Since $\varepsilon'>0$ there is a $j\in\{1,2,\ldots,n\}$ such that $a_j d_j-b_j c_j\neq 0$ and thus one of the four coefficients must be non-zero. Say $a_j$ is non-zero. Then we replace $d_j$ in $Y_{trig}$, the Fourier approximation of $\widetilde{Y}$, by $(d_j)'\defeq\tfrac{-\varepsilon'+2\pi j d_j a_j}{2\pi j a_j}$ (Step 7). If $a_j=0$, there is a different coefficient out of $b_j$, $c_j$ and $d_j$ that is non-zero and there is an analogous replacement procedure. By abuse of notation we still refer to the resulting new pair of functions as $(X_{trig}(t),Y_{trig}(t))$. Using this new finite Fourier series to define $Z_{trig}$ (Step 8) results by Lemma \ref{lem:balance} in a balanced parametrisation and a $2\pi$-periodic function $Z_{trig}$ (again abusing notation).

Step 9 is then simply the application of the inverse of the radial projection map to the curve $(X_{trig}(t),Y_{trig}(t),Z_{trig}(t))$.

\begin{lemma}
The curve given by $(x_1(t),y_1(t),x_2(t),y_2(t))$ with $t$ going from $0$ to $2\pi$ is a real-analytic Legendrian representative of $L$.
\end{lemma}
\begin{proof}
Since the functions $X_{trig}$, $Y_{trig}$ and $Z_{trig}$ are given by finite Fourier series, they are real-analytic. The projection map $(X_{trig}(t),Y_{trig}(t),Z_{trig}(t))\mapsto (x_1(t),y_1(t),x_2(t),y_2(t))$ given explicitly in Step 9 of the outline of Algorithm 1 is also real-analytic. Thus $(x_1(t),y_1(t),x_2(t),y_2(t))$ is real-analytic.

The functions $X_{trig}$, $Y_{trig}$ and $Z_{trig}$ were constructed such that $Z_{trig}+X_{trig}Y_{trig}'-Y_{trig}X_{trig}'=0$, which by Proposition \ref{thm:yuv} implies that $(x_1(t),y_1(t),x_2(t),y_2(t))$ parametrises a Legendrian knot.

By Lemma \ref{lem:iso} we only need to show that $(X_{trig}(t),Y_{trig}(t),Z_{trig}(t))$ can be taken to be an $\varepsilon$-approximation of $(\widetilde{X}(t),\widetilde{Y}(t),\widetilde{Z}(t))$ in the $C^1$-norm. The sum of the first $n$ terms of the Fourier series of $\widetilde{X}+\rmi \widetilde{Y}$ converges in the $C^1$-norm to $\widetilde{X}+\rmi \widetilde{Y}$ as $n$ goes to infinity and $X_{trig}(t)+\rmi Y_{trig}(t)$ differs from this function (at most) in one term only. In the Step 7 the coefficient $d_j$ has been replaced by $d_j'=\tfrac{\varepsilon'-2\pi j d_j a_j}{2\pi j a_j}$ (or analogous substitution for $a_j$, $b_j$ or $c_j$). We have $|d_j-d_j'|=\tfrac{\varepsilon'}{2\pi j |a_j|}$, which can be made arbitrarily small by choosing $n$ (the degree of the Fourier approximation) large enough. It follows that $X_{trig}$ and $Y_{trig}$ can be constructed to be arbitrarily close approximations of $\widetilde{X}$ and $\widetilde{Y}$, respectively. The corresponding statement for $Z_{trig}$ and $\widetilde{Z}$ follows from their definitions. It follows from Lemma~\ref{lem:iso} that the curve is ambient isotopic to $L$.
\end{proof}

Note that while the algorithm produces real-analytic parametrisations of Legendrian links of given link type, we do not consider different Legendrian link types of the same topological link type here. That is to say, at the moment there is no controlled way to construct one specific real-analytic Legendrian link type (out of the infinitely many Legendrian link types that represent a given topological link type).

\section{Example: The figure eight knot}\label{sec:ex}
In this section we use our construction to obtain a trigonometric parametrisation of a Legendrian link that is ambient isotopic to the figure-eight knot. As a consequence, the resulting  link on $S^3$ is a totally tangential representative of the isotopy class of the figure-eight knot. A mathematica file with the complete computations is available on the author's webpage \cite{page}.

A parametrisation of a knot diagram of the figure-eight knot is given by 
\begin{equation}\label{eq:41curve}
(X(t),Y(t))=(\cos(3t)(2+\cos(2t)),\sin(4t)+1/4 \sin(2t)),
\end{equation}
where $t$ is going from $0$ to $2\pi$. This particular parametrisation is found by taking a lemniscate parametrisation of the figure-eight knot as a starting point \cite{lemniscate}.
The parametric curve and the knot diagram are displayed in Figure~\ref{fig:diagrams41}a) and b).

\begin{figure}
\centering
\labellist
\Large
\pinlabel a) at 30 2800
\pinlabel b) at 30 1900
\pinlabel c) at 30 950
\endlabellist
\includegraphics[height=11.5cm]{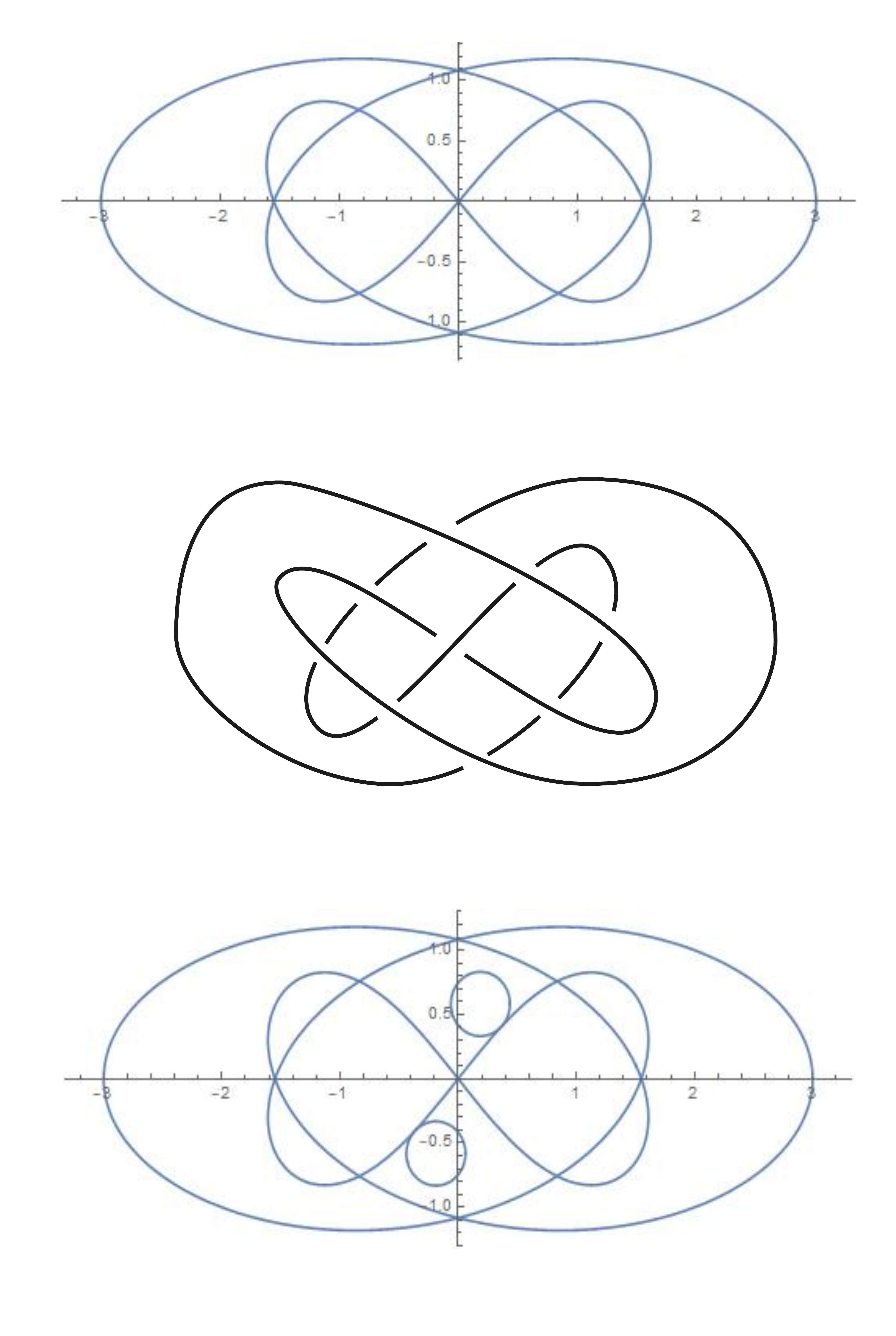}
\caption{a) The planar curve parametrised by Eq.~\eqref{eq:41curve}. b) A knot diagram of the figure-eight knot. c) The union of the planar parametric curve and the inserted circles $C_1$ and $C_2$. \label{fig:diagrams41}}
\end{figure}

We define $Z(t)=\int_{0}^tY(s)X'(s)-X(s)Y'(s)\rmd s$. For every crossing we check if $Z(t)$ induces the correct crossing sign. For example, the crossing at $(0,\tfrac{5\sqrt{3}}{8})$ obtained at $t=\tfrac{\pi}{6}$ and $t=\tfrac{7\pi}{6}$ has the correct sign, because 
\begin{equation}
Z\left(\tfrac{\pi}{6}\right)=-\tfrac{7493}{1260}<\tfrac{7493}{1260}=Z\left(\tfrac{7\pi}{6}\right).
\end{equation}

In fact, all crossings have the correct sign except for the crossing at $(0,0)$, which occurs at $t=\tfrac{\pi}{2}$ and $t=\tfrac{3\pi}{2}$. Following the algorithm we have to insert small circles in the diagram as shown in Figure~\ref{fig:diagrams41}c). Since $Z(0)=0=Z(2\pi)$, the two circles $C_1$ and $C_2$ before and after $t=\tfrac{\pi}{2}$ are the only ones that need to be inserted. We may for example pick $\tau_1=1.44996$ and $\tau_1'=1.69164$ (which are chosen conveniently because at both of these points the tangent to $(X(t),Y(t))$ has slope 1). As can be seen in Figure~\ref{fig:diagrams41}c) the value $\tau_1$ lies between the crossing at $(0,0)$ and the previous crossing (with respect to the orientation determined by the parametrisation) and $\tau_1'$ lies between the crossing at $(0,0)$ and the next crossing.

The circles $C_1$, $C_2$ are parametrised as
\begin{align}
&(c_{1,1}(t),c_{1,2}(t)):=\left(\frac{1}{4}\cos(-t+\varphi_1),\frac{1}{4}\sin(-t+\varphi_1)\right)+(U(\tau_1),V(\tau_1))+N(\tau_1)\nonumber\\
=&\left(\frac{1}{4}\cos(-t+3\pi/4)+U(\tau_1)-\frac{1}{4}\cos(3\pi/4),\frac{1}{4}\sin(-t+3\pi/4)+V(\tau_1)-\frac{1}{4}\sin(3\pi/4)\right),\\
&(c_{2,1}(t),c_{2,2}(t)):=\left(\frac{1}{4}\cos(-t+\varphi_1'),\frac{1}{4}\sin(-t+\varphi_1')\right)+(U(\tau_1'),V(\tau_1'))+N(\tau_1')\nonumber\\
&=\left(\frac{1}{4}\cos(t+7\pi/4)+U(\tau_1')-\frac{1}{4}\cos(7\pi/4),\frac{1}{4}\sin(t+7\pi/4)+V(\tau_1')-\frac{1}{4}\sin(7\pi/4))\right),
\end{align}
respectively, with $t$ going from 0 to $2\pi$. The vectors $N(\tau_1)$ and $N(\tau_1')$ are normal vectors of $(X(\tau_1),Y(\tau_1))$ and $(X(\tau_1'),Y(\tau_1'))$, respectively, of unit length, whose direction is determined by the desired positioning of the circles $C_1$ and $C_2$. For example, if the circle $C_1$ is supposed to be on the left of the curve $(X(t),Y(t))$ as in Definition~\ref{def:leftof}, then $N(\tau_1)$ should point to the left of the curve (with respect to the orientation induced by the parametrisation). The angles $\varphi_1$ and $\varphi_1'$ are the angles of $N(\tau_1)$ and $N(\tau_1')$, respectively, when interpreted as points in $\mathbb{R}^2$ in polar coordinates (with the center of the circle as the origin). By construction the circles $C_1$ and $C_2$ intersect the parametric curve $(X(t),Y(t))$ tangentially at $t=\tau_1$ and $t=\tau_1'$, respectively, and both have radius $1/4$. 

By Lemma~\ref{lem:circle} traversing the circle $C_1$ once would result in an increase of the $Z$-coordinate of $\pi/8$. Likewise, traversing $C_2$ once decreases the $Z$-coordinate by $\pi/8$. In order to obtain the correct crossing signs, $C_1$ and $C_2$ have to be traversed 
\begin{align}
m&=\left\lfloor\left|Y(\pi/2)-Y(3\pi/2)\right|8/\pi\right\rfloor+1\nonumber\\
&=\left\lfloor\left|-10.06032-10.06032\right|8/\pi\right\rfloor+1=52
\end{align} 
times each.

This way the concatenation of arcs of the original curve $(X(t),Y(t))$ and $C_1$ and $C_2$ forms a loop that is the image of a $C^1$-curve. It has the piecewise parametrisation: 

\begin{align}
\widetilde{X}(t)&\defeq\begin{cases}X(t) & \text{if }t\in[0,\tau_1]\\
c_{1,1}(t-\tau_1) & \text{if } t\in[\tau_1,\tau_1+2m\pi]\\
X(t-\tau_1-2m\pi ) & \text{if } t\in[\tau_1+2m\pi, \tau_1'+\tau_1+2m\pi]\\
c_{2,1}(t-\tau_1'-\tau_1-2m\pi) & \text{if } t\in[\tau_1'+\tau_1+2m\pi,\tau_1'+\tau_1+4m\pi]\\
X(t-\tau_1'-\tau_1-4m\pi) & \text{if } t\in[\tau_1'+\tau_1+4m\pi,2\pi(2m+1)]\end{cases}\nonumber\\
\widetilde{Y}(t)&\defeq\begin{cases}Y(t) & \text{if }t\in[0,\tau_1]\\
c_{1,2}(t-\tau_1) & \text{if } t\in[\tau_1,\tau_1+2m\pi]\\
Y(t-\tau_1-2m\pi ) & \text{if } t\in[\tau_1+2m\pi, \tau_1'+\tau_1+2m\pi]\\
c_{2,2}(t-\tau_1'-\tau_1-2m\pi) & \text{if } t\in[\tau_1'+\tau_1+2m\pi,\tau_1'+\tau_1+4m\pi]\\
Y(t-\tau_1'-\tau_1-4m\pi) & \text{if } t\in[\tau_1'+\tau_1+4m\pi,2\pi(2m+1)]\end{cases}\nonumber\\
\widetilde{Z}(t)&\defeq \int_0^t\widetilde{Y}(s)\widetilde{X}'(s)-\widetilde{X}(s)\widetilde{Y}'(s)\rmd s.
\end{align}
Note that $\widetilde{Z}$ satisfies $\widetilde{Z}(t)=Z(t)$ for all $t\in[0,\tau_1]$, $\widetilde{Z}(t)=Z(t-2m\pi)+m\pi/8$ for all $t\in[\tau_1+2m\pi,\tau_1'+\tau_1+2m\pi]$, and $\widetilde{Z}(t)=Z(t-4m\pi)$ for all $t\in[\tau_1'+\tau_1+4m\pi,2\pi(2m+1)]$.

By construction $(\widetilde{X}(t),\widetilde{Y}(t),\widetilde{Z}(t))$ is a parametrisation of the figure-eight knot. It is the image of a $C^1$-curve, but not necessarily smooth. In particular, this parametrisation is not real-analytic, it is only piecewise real-analytic.





        

After rescaling the $t$-variable, we may take all functions to be $2\pi$-periodic. Using Mathematica we can calculate a Fourier expansion of $\widetilde{X}$ and $\widetilde{Y}$ that is a sufficiently $C^1$-close approximation. In particular, if sufficiently many Fourier terms are included the Fourier series $X_{trig}$ and $Y_{trig}$ parametrise a knot diagram of the figure-eight knot and $Z_{trig}(t)$, defined as usual via $Z_{trig}=\int_0^tY_{trig}(s)X_{trig}'(s)-X_{trig}(s)Y_{trig}'(s)\rmd s$, gives the desired crossing signs. In our example, we need a Fourier sum of degree 185. 

A priori, $Z_{trig}$ is not necessarily $2\pi$-periodic, i.e., $(X_{trig}(t),Y_{trig}(t))$ is not necessarily a balanced parametrisation. In our example we have $Z_{trig}(2\pi)\approx-0.00087\neq 0=Z_{trig}(0)$. But as explained in the previous section we can fix this by changing one of the Fourier coefficients of $X_{trig}$ or $Y_{trig}$. In our example, changing the coefficient $d_{185}$ to $(d_{185})':=\tfrac{-Z_{trig}(2\pi)+370\pi d_{185}a_{185}}{370\pi a_{185}}$ gives the desired property for $Z_{trig}$. Note that this implies that $Z_{trig}(t)$ is also a (real-valued) finite Fourier sum, i.e., a trigonometric polynomial (of degree $370$). Thus $(X_{trig}(t),Y_{trig}(t),Z_{trig}(t))$, $t\in[0,2\pi]$, is a real-analytic parametrisation of a figure-eight knot in $\mathbb{R}^3$ and its image in $S^3$ is a real-analytic Legendrian figure-eight knot with respect to the standard contact structure.

Figure~\ref{fig:fig8knot} shows the curve in $\mathbb{R}^3$ and the planar curve $(X_{trig}(t),Y_{trig}(t))$, $t\in[0,2\pi]$. As we can see it is a very close approximation of the original diagram curve $(X(t),Y(t))$, $t\in[0,2\pi]$, in Figure~\ref{fig:diagrams41}

\begin{figure}
\includegraphics[height=10cm]{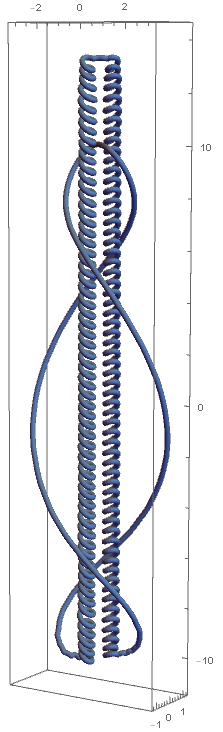}
\includegraphics[height=2.5cm]{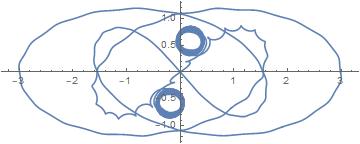}
\caption{The figure-eight knot that is the result of the algorithm and its projection in the plane.\label{fig:fig8knot}}
\end{figure}

\section{A variation of the algorithm}\label{sec:variations}
Even though the figure-eight knot is a comparatively simple knot, we needed 370 Fourier terms in our parametrisation in the previous section. We will see later that the degree of the trigonometric polynomials that parametrise a knot determine the size of the system of linear equations that needs to be solved in order to find the holomorphic functions that would produce examples of totally-tangential $\mathbb{C}$-links or electromagnetic knots. We should therefore try to adapt Algorithm 1, so that it produces parametrisations in terms of trigonometric polynomials of lower degree.

A variation of the algorithm above takes place in $\mathbb{R}^3$ with its standard contact structure $\xi_{1}$ defined via $\alpha_1=\rmd z+x\rmd y$ as opposed to the (equivalent) contact structure $\xi_2$ defined via $\rmd Z+X\rmd Y-Y\rmd X$, which came from the projection of the standard contact structure on (the upper hemisphere of) $S^3$. As in \cite[Example 2.1.3]{geiges}, there is a real-analytic diffeomorphism
\begin{equation}
\varphi(x,y,z)=\left(\frac{x+y}{2},\frac{y-x}{2},z+\frac{xy}{2}\right)
\end{equation}
that maps $(\mathbb{R}^3,\xi_1)$ to $(\mathbb{R}^3,\xi_2)$. Thus if we can find a real-analytic Legendrian knot in $(\mathbb{R}^3,\xi_1)$, we can map it to $(\mathbb{R}^3,\xi_2)$ and then to $S^3$ to obtain the desired space curve there.

Every Legendrian knot $(x(t),y(t),z(t))$ in $(\mathbb{R}^3,\xi_1)$ can be represented by a front projection $(y(t),z(t))$, a type of knot diagram without vertical tangencies, where the $x$-coordinate (the direction outside of the diagram plane) is determined by $x(t)=-z'(t)/y'(t)$. Searching real-analytic Legendrian knots thus means essentially looking for real-analytic front projections.

Since the $x$-coordinate is determined by the slope $-z'/y'$, the sign of a crossing in a front projection is such that the overpassing strand has smaller slope. (Note that there is no definite convention on the definition of the contact form of the standard contact structure on $\mathbb{R}^3$. Practically any choice of signs and permutations of the coordinates can be found in the literature. In particular, for other forms such as $\rmd z-x\rmd y$, it is the larger slope that belongs to the overpassing strand. We use the definitions from \cite{geiges}.)

As in the original algorithm we may start with a parametrisation of a given knot type, found for example via the methods in \cite{bodepoly, semiAK}. In general, the $y$- and $z$-coordinates of the curve do not provide a front projection. There might be vertical tangencies and the crossing signs induced by the slope might be different from the desired crossing signs.

We may delete a neighbourhood of any point vertical tangent from the diagram and replace it with a parametric arc that has a cusp instead of the vertical tangent and that extends the diagram curve to a $C^1$ piecewise-smooth curve. In Mathematica such parametric curves can be found with the command InterpolatingPolynomial. Note that any trigonometric polynomial can be interpreted as a complex polynomial restricted to the unit circle. We only have to specify the coordinates of the endpoints of the diagram curve (once the neighbourhood of the point with vertical tangent is removed), the values of the derivatives of the coordinate functions at these endpoints and the desired values of the interpolating polynomial and its first three derivatives at the cusp point. As the name of the command suggests, each coordinate of the resulting parametrisation is given by a polynomial with specified values on points $\rme^{\rmi t}$ on the unit circle, which corresponds to a trigonometric polynomials with specified values on points in the interval $[0,2\pi]$.

Likewise, if a crossing has the ``wrong'' sign in the sense that we require a strand with smaller slope to pass over a strand with bigger slope, we can delete a small arc in a neighbourhood of the crossing and replace it by a ``zigzag''. This arc has two cusps and is chosen such that it has a unique crossing with the rest of the diagram and the sign of the crossing is the desired one. As in the previous case, it should extend the diagram curve to a $C^1$ curve. Again we may use the same command InterpolatingPolynomial. This procedure is not quite algorithmic in the sense that it is not immediate that the inserted arc does not intersect the rest of the diagram in more points.

An example for the figure-eight knot is shown in Figure~\ref{fig:frontproj}. Here, the signs of the crossings as induced by the slope are correct and produce the figure-eight knot. Therefore, it is not necessary to introduce zigzags. However, we need to replace the vertical tangent points with cusps.

\begin{figure}
\centering
\labellist
\Large
\pinlabel a) at 0 330
\pinlabel b) at 450 330
\endlabellist
\includegraphics[height=4cm]{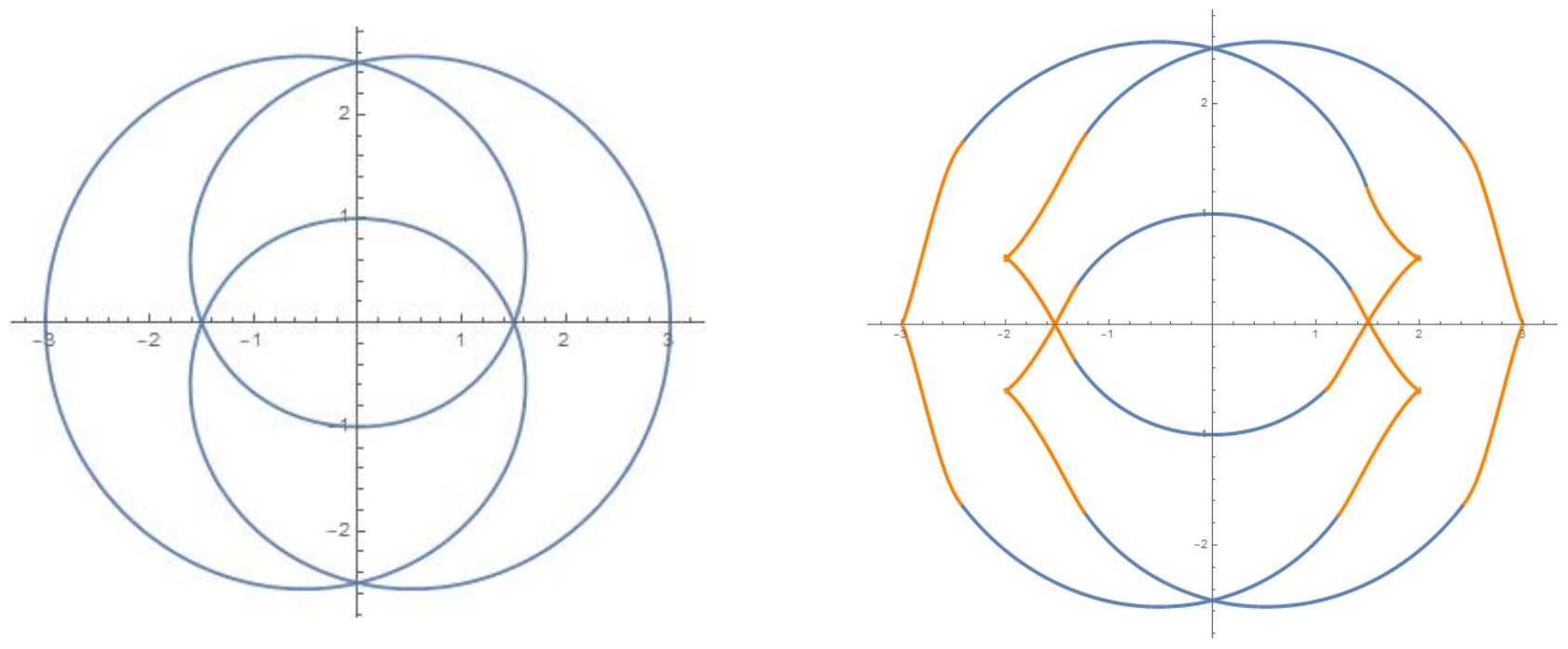}
\caption{a) A projection of the figure-eight knot into the plane, i.e., a knot diagram without specified crossing signs. b) A planar curve that corresponds to a front projection of the figure-eight knot, defined piecewise via a trigonometric parametrisation of a knot diagram (in blue) and interpolating trigonometric polynomials (in orange).\label{fig:frontproj}}
\end{figure}

We thus obtain a $C^1$, piecewise smooth diagram curve $(y(t),z(t))$ such that $x(t)=-z'(t)/y'(t)$ defines the $x$-coordinate of a $C^1$, piecewise smooth parametrisation $(x(t),y(t),z(t))$, $t\in[0,2\pi]$, of a Legendrian knot in $(\mathbb{R}^3,\xi_1)$.

We can apply the map $\varphi$ to obtain a parametrisation of a Legendrian knot in $(\mathbb{R}^3,\xi_2)$ in coordinates $(X,Y,Z)$, approximate $(X(t),Y(t))$ by finite Fourier sums $(X_{trig}(t),Y_{trig}(t))$ and define $Z_{trig}(t)$ as in the previous sections. We can then follow Steps 7-9 from the algorithm to obtain a real-analytic Legendrian knot in $S^3$.

One reason why we can perhaps expect this method to result in lower Fourier degrees is the following argument. The original algorithm required the insertion of several small circles into the diagram in order to ``wind up'' or ``wind down'' the $Z$-coordinate. These circles had to be rather small so that they do not intersect the rest of the diagram. This implies that they have to be traversed many times in order to increase or decrease the $Z$-coordinate by the sufficient amount, resulting in significant contributions to Fourier terms of large degree. Since the variation of the algorithm using the front projection does not require the insertion of such small circles, we might be able to get away with fewer terms.

In practice, this method has its disadvantages though. In the original algorithm we only needed to calculate the Fourier transforms of functions that were piecewise trigonometric polynomials. Here we are dealing with piecewise functions of quotients of trigonometric polynomials, which are a lot harder to compute numerically. (We let such a computation for the figure-eight knot run for several days on a standard laptop and eventually gave up.)

Another big issue is the numerical accuracy in the definition of $x(t)=-z'(t)/y'(t)$. In a neighbourhood of the cusp, $(y(t),z(t))$ are interpolating trigonometric polynomials that are chosen such that they locally look like $y^3\pm z^2=0$. In particular, at the cusp we require that $y'(t)=z'(t)=0$. By prescribing the values of the higher order derivatives of $y$ and $z$ at that point, we can guarantee that $x$ is well-defined. In practice however, the interpolation is only achieved to the degree of numerical accuracy. Instead of $y'(t)=z'(t)=0$, we might have functions with $y'(t)\approx10^{-22}$ and $z'(t)\approx10^{-20}$, resulting in $|x(t)|\approx100$ instead of $x(t)=0$, or $y'(t)=0$ and $z'(t)\approx 10^{-20}$, resulting in an error, since the definition of $x$ requires division by $0$.

In this sense, the algorithm from the previous sections is a lot more reliable, even though it produces parametrisations of very high degree. In practice, we may look at the piecewise parametrisation of the front projection and attempt to find a Fourier approximation in coordinates $(X,Y,Z)$ by trial and error. This way we obtained the following parametrisation of the figure-eight knot, which has degree 11:
\begin{align} \label{eq:best_fig8}
X(t)&=2.4 \cos(3t)+\cos(t)\nonumber\\
Y(t)&=(1.5 \cos(t)^2 + \tfrac{2}{3} \sin(t)^2) (1.5 \cos(6 t - 1) + \sin(6 t))\nonumber\\
Z(t)&=-1.4183 \cos(t) - 3.3015 \cos(3 t) - 1.1110 \cos(5 t) -  0.4511 \cos(7 t)\nonumber\\
& - 0.4169 \cos(9 t) - 0.0921 \cos(11 t) - 3.9589 \sin(t) - 9.2153 \sin(3 t)\nonumber\\
& - 3.1011 \sin(5 t) - 1.2590 \sin(7 t) - 1.1636 \sin(9 t) - 0.2571 \sin(11 t),
\end{align}
where the coefficients of $Z(t)$ are rounded. Recall that the algorithm produced a Fourier parametrisation of degree 370. We were able to reduce the degree significantly. However, the procedure is not really algorithmic and becomes less feasible for more complicated knots.

Figure~\ref{fig:fig8constr} shows the constructed curve.

\begin{figure}[H]
\centering
\includegraphics[height=4cm]{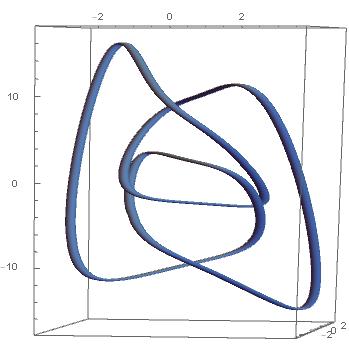}
\caption{The constructed real-analytic Legendrian figure-eight knot of Fourier degree 11.\label{fig:fig8constr}}
\end{figure}

It is worth mentioning that the search for "simple" trigonometric parametrisations of knots is of independent interest \cite{lissajous, fourier}. Here we are dealing with the additional challenge that we require the knots to be Legendrian.

\section{Systems of linear equations for totally tangential $\mathbb{C}$-links and electromagnetic knots}\label{sec:linear}

We prove Theorem~\ref{thm:2}, which asserts that we can find system of linear equations, whose solutions produce candidates for explicit examples of the functions defining totall tangential $\mathbb{C}$-links and electromagnetic knots, respectively. In the case of totally tangential knots, solutions to the equations always exist, but the corresponding functions do not necessarily satisfy all conditions in Definition~\ref{def:tt}. For electromagnetic knots, we show that the solutions do satisfy all necessary properties, but it is not known if the systems always have solutions.

\begin{proof}[Proof of Theorem~\ref{thm:2}:]
$i)$ We assume that $L$ is a knot. Since $L$ is an output of Algorithm 1, it is a Legendrian link with respect to the standard contact structure on $S^3$ and is given as a parametric curve via $(x_1(t),y_1(t),x_2(t),y_2(t))=(\tfrac{1}{R},\tfrac{X_{trig}(t)}{R},\tfrac{Y_{trig}(t)}{R},\tfrac{Z_{trig}(t)}{R})$, where $R=\sqrt{1+X_{trig}^2(t)+Y_{trig}^2(t)+Z_{trig}^2(t)}$ and $X_{trig}$, $Y_{trig}$ and $Z_{trig}$ are real trigonometric polynomials.

Consider now a polynomial $G:\mathbb{C}^2\to\mathbb{C}$, $G(z_1,z_2)=\sum_{i,j=0}^nc_{i,j}z_1^iz_2^j$, $c_{i,j}\in\mathbb{C}$, and evaluate it on $L$. We write $z_j(t)=x_j(t)+\rmi y_j(t)$, $j=1,2$, for the parametrisation of $L$ in complex coordinates. Then $G|_L=0$ if and only if 
\begin{align}\label{eq:linear}
\tilde{G}(t)&:=R^{2n}G(z_1(t),z_2(t))\nonumber\\
&=\sum_{i,j=0}^nc_{i,j}(1+\rmi X_{trig}(t))^i(Y_{trig}(t)+\rmi Z_{trig}(t))^jR^{2n-i-j}=0
\end{align}
for all $t\in[0,2\pi]$.

Suppose now that $c_{i,j}=0$ for all $i,j$ with odd $i+j$. Then $\tilde{G}(t)$ is a trigonometric polynomial with respect to the variable $t$. Suppose that $m$ is the maximum of the degrees of the trigonometric polynomials $X_{trig}$, $Y_{trig}$ and $Z_{trig}$. Then the degree of the term $z_1(t)^iz_2(t)^jR^{2n-i-j}$ is at most $mi+mj+m(2n-i-j)=2mn$. Thus the degree of the trigonometric polynomial $\tilde{G}(t)$ is at most $2mn$. 

The space of trigonometric polynomials of degree at most $2mn$ is a vector space with basis $1$, $\cos(t)$, $\sin(t)$, $\cos(2t)$, $\sin(2t)$, ... $\cos(2mn)$, $\sin(2mn)$. Expressing $\tilde{G}(t)$ in this basis gives
\begin{equation}
\tilde{G}(t)=A_0+\sum_{k=1}^{2mn}A_k\cos(kt)+B_k\sin(kt),
\end{equation}
where each of the coefficients $A_k$, $B_k$ is a linear expression in the $c_{i,j}$s. Note that Eq.~\eqref{eq:linear} is satisfied if and only if all of these coefficients $A_k$, $k=0,1,2,\ldots,2mn$, $B_k$, $k=1,2,\ldots,2mn$, are zero. 

Therefore, the polynomial $G$ vanishes on $L$ if its coefficients $c_{i,j}$ are equal to 0 for odd values for $i+j$ and the coefficients $c_{i,j}$ with even values of $i+j$ solve a homogeneous system of $4mn+1$ linear equations. The matrix $A$ mentioned in the statement of the theorem is thus the matrix corresponding to this system of linear equations.

By definition $G$ is a polynomial of degree at most $2n$. The coefficient of the unique monomial that realises that degree is $c_{n,n}$. There are thus $(n+1)^2$ coefficients $c_{i,j}$. We know that by assumption all coefficients $c_{i,j}$ with odd $i+j$ vanish. This leaves us with $\left\lceil\frac{(n+1)^2}{2}\right\rceil$ tuples $(i,j)$ with $i+j$ even, where $\lceil\cdot\rceil$ is the ceiling function that maps a number $x$ to the smallest integer that is greater than or equal to $x$. Since these coefficients $c_{i,j}$ with $i+j$ even are the variables in the system of linear equations, we have $\left\lceil\frac{(n+1)^2}{2}\right\rceil$ variables.

For sufficiently large values of $n$, specifically, when $n>\lfloor4m-1+\sqrt{2}\sqrt{4m-1+8m^2}\rfloor$, the number of variables $\left\lceil\frac{(n+1)^2}{2}\right\rceil$ becomes strictly larger than the number of linear equations $4mn+1$. It follows from the Rank-Nullity-Theorem that for these values of $n$ there are non-trivial solutions to the matrix equation. Since the solutions are coefficients $c_{i,j}$ of $G$, this implies that for every sufficiently large $n$ there is a complex polynomial (that is not constant 0) of degree at most $n$ that vanishes on $L$.

If $L$ has multiple components, given in terms of a set of parametric curves from Algorithm 1, we can run through the same argument for each component. The desired matrix $A$ is the result of stacking the rows from the matrices from the different components on top of each other, so that the size of the matrix is $(1+4n\sum_{k=1}^\ell m_k)\times \left\lceil\frac{(n+1)^2}{2}\right\rceil$, where $\ell$ is the number of components of $L$ and $m_k$ is the maximum degree of $X_{trig}$, $Y_{trig}$ and $Z_{trig}$ of the $k$th component of $L$.

$ii)$ Again we first focus on one component. The following idea already appears in \cite{bode}. Let $L$ is a real-analytic Legendrian link and $X_{(z_1,z_2)}$, $(z_1,z_2)\in L$ a non-vanishing section of the tangent bundle of $L$. Then we can express the tangent vector $X_{(z_1,z_2)}$ at $(z_1,z_2)$ in terms of the basis 
\begin{align}
v_1&=-x_2\partial_{x_1}+y_2\partial_{y_1}+x_1\partial_{x_2}-y_1\partial_{y_2},\\
v_2&=-y_2\partial_{x_1}-x_2\partial_{y_1}+y_1\partial_{x_2}+x_1\partial_{y_2}
\end{align}
of the contact plane at that point and define the complex-valued function $H:L\to\mathbb{C}$ via 
\begin{equation}
H(z_1,z_2)=X_{(z_1,z_2)}\cdot v_2+\rmi X_{(z_1,z_2)}\cdot v_1.
\end{equation} 
Any holomorphic function $h:\mathbb{C}^2\to\mathbb{C}$ with $h|_L=H$ defines a Bateman field of Hopf type that has $L$ as a set of closed field lines of its magnetic part \cite{bode}. Small variations of this argument allow for different components of $L$ to form magnetic field lines and other components to form electric field lines \cite{bode}.

Suppose now that $L$ is an output of Algorithm 1 and $h:\mathbb{C}^2\to\mathbb{C}$, $h(z_1,z_2)=\sum_{i,j=0}^nc_{i,j}z_1^iz_2^j$ is a complex polynomial. Using the same idea as above in $i)$, we want to find the coefficients $c_{i,j}$ such that the equation
\begin{equation}\label{eq:heq}
\tilde{h}(t):=R^{2n}h(z_1(t),z_2(t))=R^{2n}H(z_1(t),z_2(t))
\end{equation} 
is satisfied for all $t\in[0,2\pi]$. If we choose 
\begin{equation}
X_{(z_1(t),z_2(t))}:=x_1'(t)\partial_{x_1}+y_1'(t)\partial_{y_1}+x_2'(t)\partial_{x_2}+y_2'\partial_{y_2},
\end{equation}
then $H$ is (as a function of $t$) the quotient of a trigonometric polynomial and $R^4$, so that the right hand side of Eq.~\eqref{eq:heq} is a trigonometric polynomial if $n\geq 2$. Likewise, if we assume that all coefficients $c_{i,j}$ with $i+j$ odd are equal to zero, then the left hand side becomes the trigonometric polynomial
\begin{equation}
\tilde{h}(t)=\underset{i+j\text{ even}}{\sum_{i,j=0}^n}c_{i,j}(1+\rmi X_{trig}(t))^i(Y_{trig}(t)+\rmi Z_{trig}(t))^jR^{2n-i-j}.
\end{equation}
Once again, we obtain in Eq.~\eqref{eq:heq} an equation of trigonometric polynomials. Each basis element in the vector space of trigonometric polynomials of a certain bounded degree thus results in an inhomogeneous linear equation in the variables $c_{i,j}$. Note that the matrix corresponding to this system of linear equations is exactly the matrix $A$ from part $i)$ above.

We write $y$ for the vector that expresses the trigonometric polynomial $R^{2n}H(z_1(t),z_2(t))$ in terms of the basis elements as in $i)$. We take the coefficients $c_{i,j}$ with $i+j$ even as coordinates of a vector $x$. Then the above implies that the coordinates of solutions to $Ax=y$ are the coefficients $c_{i,j}$ of a complex polynomial $h$ that defines a Bateman field of Hopf type of the desired topology.

Again, for multiple components we obtain a set of linear equations for each component, which can be combined into one larger system of linear equations. The difference to $i)$ is that now the system of linear equations is inhomogeneous, so in general we do not know if there is a choice of $n$ such that a solution exists.
\end{proof}

The relation between the holomorphic functions $G$ and $h$ that define a real-analytic Legendrian link as a totally-tangential $\mathbb{C}$-link and an electromagnetic knot, respectively, has already been the object of speculation in \cite{bode}. Note that for the (only) known examples of torus knots, we have $h(z_1,z_2)=\tfrac{\partial^2G(z_1,z_2)}{\partial z_1\partial z_2}$. Calculations in the last section of \cite{tori} suggest that we should not expect this particular relation to hold in general. However, the proof of Theorem~\ref{thm:2} establishes a connection of a different type. If both functions are assumed to be polynomials, then their coefficients are solutions to a homogeneoues and an inhomogeneous system of linear equations, respectively, with both systems corresponding to the same matrix $A$ and the inhomogeneous part of the equation determined by the knot in question.

Recall that the Legendrian parametrisation that we found for the figure-eight knot in Section~\ref{sec:ex} has Fourier degree 370. In order to guarantee a solution to the corresponding system of linear equations we thus need $n>\lfloor4m-1+\sqrt{2}\sqrt{4m-1+8m^2}\rfloor$. With $m=370$ this results in $n>2958$. The resulting system of linear equations then consists of more than 4 milion equations in more than 4 milion variables. 

Even for the simpler parametrisation that we found in Section~\ref{sec:variations} we still have $m=11$, $n>86$ and a system of 3829 linear equations in 3872 variables. This illustrates several things. First of all, in order to construct totally tangential $\mathbb{C}$-links it is important to find Fourier parametrisations of the lowest degree possible. Secondly, the solution of the system linear equations that produces the coefficients of the desired complex plane curve $G$ requires substantial computing power. As a further consequence the degree of the found polynomial will be very large (174 for the figure-eight knot of Fourier degree 11 and 5918 for the figure-eight knot of Fourier degree 370). Considering that the figure-eight is a comparatively simple knot, this could explain why no examples beyond the torus knots have been found so far (although it is theoretically possible that a given knot arises from a polynomial of lower degree).

\section{Dynamics of electromagnetic knots}\label{sec:dynamics}

In this section we describe the motion of knots that are closed field lines of Bateman electromagnetic fields of Hopf type. At first, this might seem like a challenging problem, since we do not know any explicit expressions for these fields except the known examples of the fields of torus knots. However, we will see that it is not necessary to know the electromagnetic field to study the motion of an electromagnetic knot in a Bateman field of Hopf time. It is sufficient to know the inital curve at $t=0$, which is exactly what Algorithm 1 produces.

There is a smooth family of diffeomorphisms $\Phi_t$ of $\mathbb{R}^3$ such that if $L$ is a set of closed field lines at time $t=0$, then $\Phi_t(L)$ is a set of closed field lines at time $t$. Note that this family of diffeomorphisms, which determines the time evolution of the entire field, is the same for every Bateman field of Hopf type. This can be verified via a direct calculation or by envoking the Robinson congruence \cite{robinson, atiyah, kedia}. It is also apparent from the fact that Bateman fields in $\mathbb{R}^3$ are (up to a scalar factor) projections of time-independent vector fields on $S^3$, where the only thing that depends on time is the projection map $\varphi_{t=t_*}=(\alpha,\beta)|_{t=t_*}^{-1}$. Since the projection maps are the same for every Bateman field, they all evolve using the same ambient isotopy $\Phi_t$. The time-evolution of the link $L$ via $\Phi_t$ can be interpreted as $L$ being dragged along the normalised Poynting vector field $\mathbf{V}=\mathbf{E}\times\mathbf{B}/|\mathbf{E}\times\mathbf{B}|$, which for Bateman fields of Hopf type is equal to 
\begin{align}
\mathbf{V}=&\text{Re}(\nabla\alpha\times\nabla\beta)\times\text{Im}(\nabla\alpha\times\nabla\beta)\nonumber\\
=&\left(\frac{2(x(t-z)+y)}{x^2+y^2+z^2+t^2+1-2tz},\frac{2(y(t-z)+x)}{x^2+y^2+z^2+t^2+1-2tz},\right.\nonumber\\
&\left.\frac{x^2+y^2-z^2-t^2-1+2tz}{x^2+y^2+z^2+t^2+1-2tz}\right).
\end{align}
In other words, the trajectory of each point $p\in L$ traces out a parametric curve, whose velocity vector is the normalised Poynting vector field at its position. In particular, each point on the link moves with the speed of light. For an interpretation of $\mathbf{V}$ as the velocity field of an Euler flow, see \cite{kedia2}.

We prove that no compact subset of $\mathbb{R}^3$ can contain an electromagnetic knot indefinitely.

\begin{proof}[Proof of Theorem~\ref{thm:dynamics}]
Since $L$ is a link formed by a set of closed field lines of $\mathbf{F}$ at time $t=0$, the map $\varphi_0^{-1}=(\alpha,\beta)|_{t=0}:\mathbb{R}^3\to S^3$ maps $L$ to a real-analytic Legendrian link $\tilde{L}$ with respect to the standard contact structure. Since $(\alpha,\beta)|_{t=0}$ extends to a diffeomorphism mapping the point at infinity to $(1,0)\in S^3\subset\mathbb{C}^2$, the link $\tilde{L}$ does not pass through $(1,0)$ and so there is a neighbourhood $U$ of $(1,0)$ in $S^3$ that is disjoint from $\tilde{L}$.

For every point $(x,y,z)\in\mathbb{R}^3$ and every $t\in\mathbb{R}$ we can consider the corresponding preimage point on $S^3$, which is $\varphi_t^{-1}(x,y,z)$. We calculate the limits $\lim_{t\to\infty}\varphi_t^{-1}(x,y,z)$ and $\lim_{t\to-\infty}\varphi_t^{-1}(x,y,z)$. Since $\varphi_t^{-1}$ is the map $(\alpha,\beta)$ for a fixed value of $t$, we obtain $(1,0)\in S^3\subset\mathbb{C}^2$ in both cases, see Eq.~\eqref{eq:alphabeta}. In particular, for every compact subset $K$ of $\mathbb{R}^3$ there is a $T>0$ such that for all $t$ with $|t|>T$ the preimage set $\varphi_t^{-1}(K)$ is contained in $U$ and thus disjoint from $\tilde{L}$. Since the link $\Phi_t(L)$ that is formed by a set of closed field lines of $F$ at time $t$ is equal to $\varphi_t(\tilde{L})$ and $\varphi_t:S^3\to\mathbb{R}^3\cup\{\infty\}$ is a diffeomorphism for all $t$, we have that $K$ and $\Phi_t(L)$ are disjoint for all $t$ with $|t|>T$.
\end{proof}

We don't have an explicit analytic expression for $\Phi_t=\varphi_t\circ\varphi_0^{-1}$, but
\begin{align}
\Phi_t^{-1}&=\varphi_0\circ\varphi_t^{-1}=\varphi_0\circ(\alpha,\beta)\nonumber\\
&=\left(\frac{\text{Re}(\beta)}{1-\text{Re}(\alpha)},\frac{-\text{Im}(\beta)}{1-\text{Re}(\alpha)},\frac{\text{Im}(\alpha)}{1-\text{Re}(\alpha)}\right)\nonumber\\
&=\left(\frac{-2ty+x(x^2+y^2+z^2-t^2+1)}{x^2+y^2+z^2+t^2+1-2tz},\frac{2tx+y(x^2+y^2+z^2-t^2+1)}{x^2+y^2+z^2+t^2+1-2tz},\right.\nonumber\\
&\left.\frac{-t(x^2+y^2+z^2-t^2-1)+z(x^2+y^2+z^2-t^2+1)}{x^2+y^2+z^2+t^2+1-2tz}\right).
\end{align} 
Given a real-analytic Legendrian link $L$ in $S^3$ as a parametric curve, we obtain a parametric curve $\varphi_0(L)$, which arises as a set of closed field lines of some Bateman field of Hopf type. We can then find the preimage of $\varphi_0(L)$ under the map $\Phi_t^{-1}$ numerically, which is the set of closed field lines of the same electromagnetic field at the time $t$. We can thus plot the electromagnetic link at different moments in time without knowledge of the surrounding electromagnetic field and without having to integrate the normalised Poynting vector field.

Figures~\ref{fig:dynamics}a)-g) show such a sequence of snapshots of the motion of the constructed figure-eight knot from Eq.~\eqref{eq:best_fig8}. Note that the curve in Figure~\ref{fig:dynamics}a) is different from Figure~\ref{fig:fig8constr}, which shows the constructed figure-eight knot in $\mathbb{R}^3$. Figure~\ref{fig:dynamics} displays the curve that is the result of projecting the curve in Figure~\ref{fig:fig8constr} into $S^3$ via the inverse of a radial projection map (where $\mathbb{R}^3$ is identified with the tangent space of $S^3$ at $(1,0)$) and then projecting that curve into $\mathbb{R}^3$ via $\varphi_0=((\alpha,\beta)|_{t=0})^{-1}$, which is the stereographic projection followed by a mirror reflection along the $y=0$-plane.

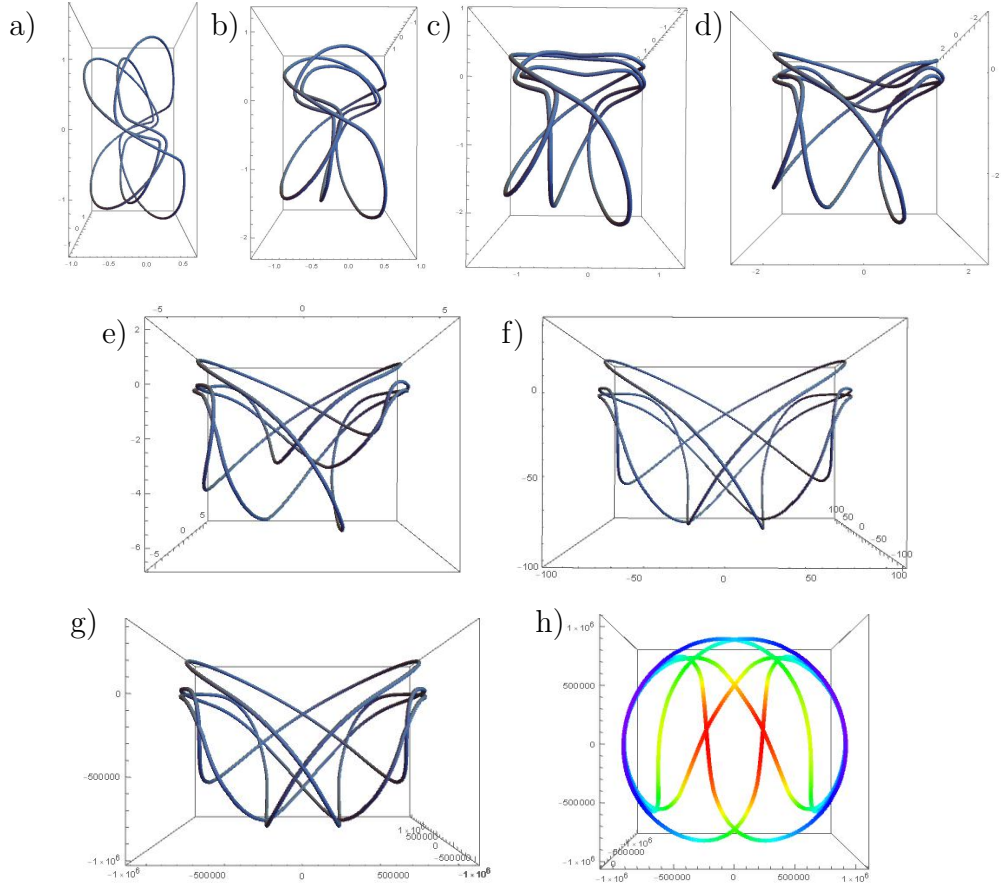
\begin{figure}

\centering
\labellist
\Large
\pinlabel a) at 0 1400
\pinlabel b) at 330 1400
\pinlabel c) at 680 1400
\pinlabel d) at 1120 1400
\pinlabel e) at 150 900
\pinlabel f) at 800 900
\pinlabel g) at 100 430
\pinlabel h) at 860 430
\endlabellist
\includegraphics[height=12cm]{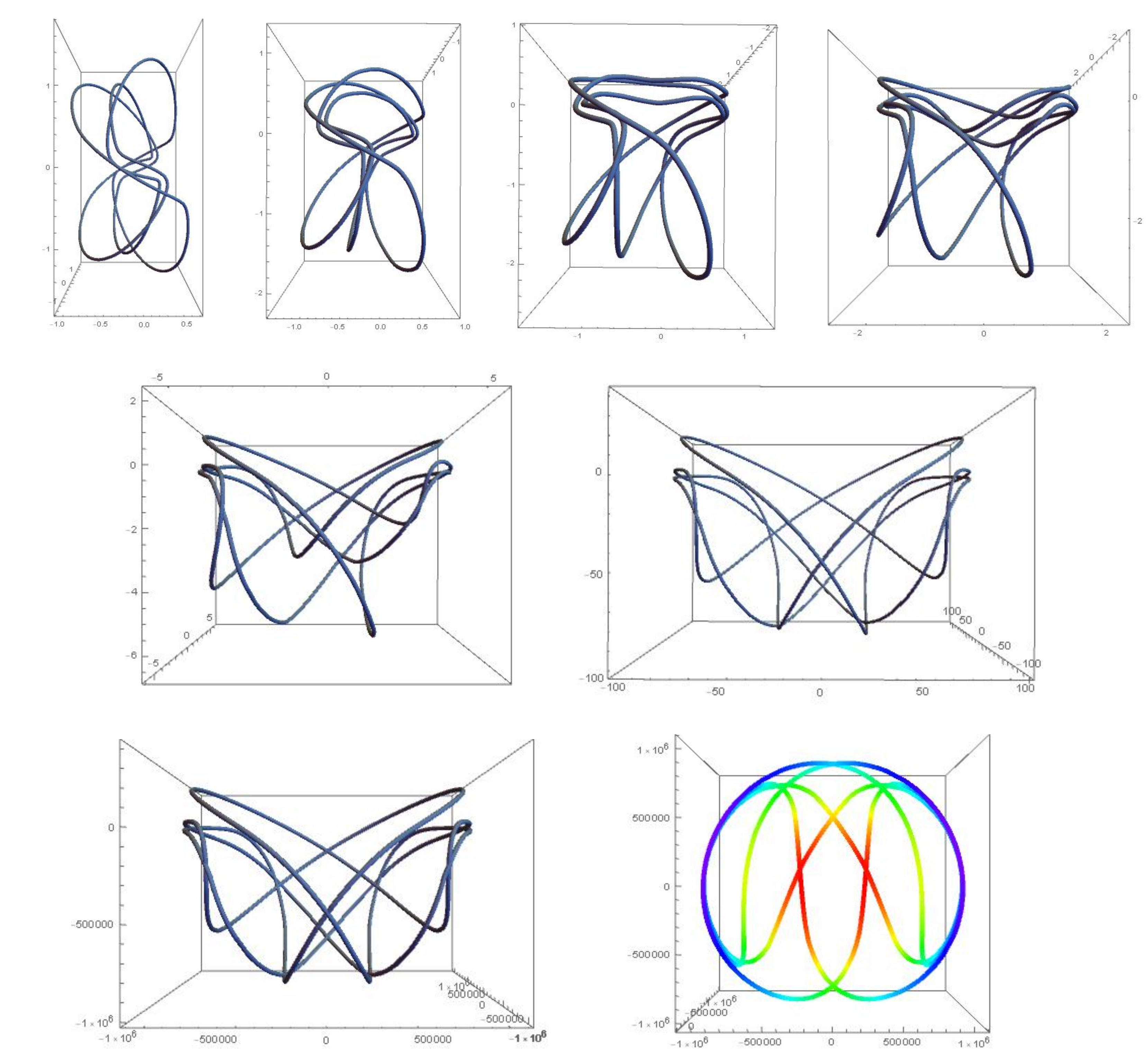}
\caption{The time evolution of an electromagnetic figure-eight knot. a) $t=0$. b) $t=0.5$. c) $t=1$. d) $t=2$. e) $t=5$. f) $t=100$. g) $t=1000000$. h) $t=1000000$ viewed from above and coloured by the $z$-coordinate with red corresponding to a very low numbers. \label{fig:dynamics}}
\end{figure}

There are two dominant aspects of the flow that allow a qualitative description of the general motion of an electromagnetic knot. The first is a downward motion in a growing region around the $z$-axis. Note that for all $t$ the normalised Poynting vector field $\mathbf{V}$ is $(0,0,-1)$, i.e., it points vertically downwards. By continuity, there is a neighbourhood of the $z$-axis, where $\mathbf{V}$ is very close to being vertical at $t=0$. Furthermore, $\mathbf{V}$ converges pointwise to $(0,0,-1)$ as $t$ approaches infinity. Thus the region where $\mathbf{V}$ points almost downwards, say $U_t:=\{(x,y,z)\in\mathbb{R}:|\mathbb{V}-U_t|<\varepsilon\}$ for some small $\varepsilon>0$, grows as $t$ increases. 

The second important contribution moves the knot to larger and larger values of $R$ in cylindrical coordinates $(R,\phi,z)$, i.e., the knot grows radially. It moves further and further away from the $z$-axis. 

For a knot $L=\Phi_0(L)$ that is initially linked with the $z$-axis, i.e., not null-homotopic in $\mathbb{R}^3\backslash\{(0,0,z):z\in\mathbb{R}\}$, a typical motion can be described as follows. The knot grows in the radial direction. (Note the labels on the axes in Figure~\ref{fig:dynamics}.) However, the region $U_t$ also grows, starting from a small neighbourhood of the $z$-axis. If $U_t$ catches up with a part of the knot $\Phi_t(L)$, the dominant direction for that part of the knot becomes downwards. Since the part of the knot $\Phi_t(L)$ that is reached first by the growing $U_t$ is the part that lies closest to the $z$-axis, this means that as $t$ approaches $\infty$, the knot develops the shape of a bowl, i.e., it is close to a curve that lies on the lower half of a sphere. As $t$ approaches $-\infty$, the knot almost lies on the upper half of a sphere.

Figure~\ref{fig:dynamics}h) shows the conformation of the figure-eight knot at $t=1000000$ viewed from above with colors indicating the values of the $z$-coordinate. Red corresponds to very low values (close to $-10^6$). Traversing the colour wheel in positive directions (orange, yellow, green, blue) corresponds to larger and larger values of $z$-coordinate. This demonstrates the ``bowl-shape'' of the figure-eight knot at large values of $t$.

This limit shape is easier to see in a different example, see Figure~\ref{fig:dynamics}. It is again a figure-eight knot, this time of Fourier degree 17. However, the initial shape is more elongated (Figure~\ref{fig:dynamics2}b)), so that the limit shape at large values of $t$ resembles a $U$ (Figure~\ref{fig:dynamics2}c)) and the bowl-shape is more easily visible. For large negative values of $t$, the knot forms an upside-down $U$, see Figure~\ref{fig:dynamics2}a).

\begin{figure}
\centering
\labellist
\Large
\pinlabel a) at 30 640
\pinlabel b) at 490 640
\pinlabel c) at 250 250
\endlabellist
\includegraphics[height=8cm]{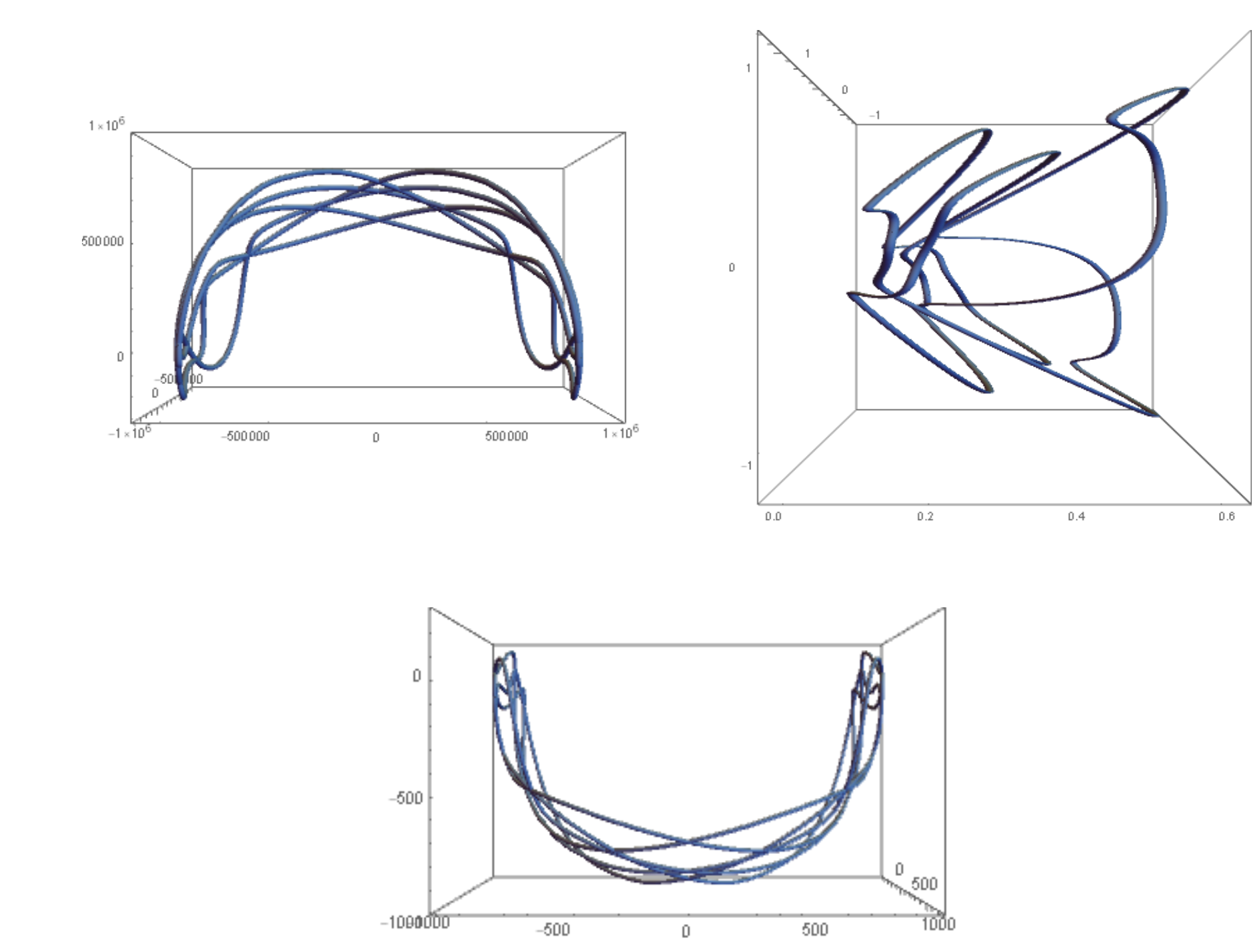}
\caption{The time evolution of an electromagnetic figure-eight knot. a) $t=-1000000$. b) $t=0$. c) $t=1000000$.\label{fig:dynamics2}}
\end{figure}

This was a quite superficial analysis of the limit shape of the knot, but it should hold in general for electromagnetic knots that are initially linked with the $z$-axis. In fact, all curves that we have investigated and that are initially linked with the $z$-axis share this pattern. They start out as very large knots (in a neighbourhood of the point at infinity) in the shape of an upside-down $U$ for large negative values of $t$, or rather, an upside-down bowl. Then they shrink until $t$ is close to $0$. For values of $t$ around 0 interesting geometric changes can occur, which depend hugely on the particular geometry of the curve. Then the knot grows again developing an overall $U$-shape or bowl shape. Both the $U$-shape for large values of $t$ and the upside-down $U$ for large negative values are also apparent in the simulations of electromagnetic torus knots in the supplementary material of \cite{kedia}, although the simulations do not go beyond relatively small values of $|t|$.

\begin{remark} 
Theorem~\ref{thm:dynamics} and the observation that knots in Bateman fields of Hopf type grow to infinity is in contrast to the time evolution of knots in plasma, where the Lorentz force and magnetic pressure cause the knot to shrink as much as possible \cite{moffatt}.

Recall that Bateman fields are solutions to Maxwell's equations in vacuum. In the absence of any charges, there is nothing that confines the knots to small regions, resulting in the growth discussed above.
\end{remark}

\ \\

\noindent
\textbf{Acknowledgments:} The author is grateful to Lee Rudolph and Daniel Peralta Salas for fruitful discussions and helpful explanations and to Oliver Gross and Mitchell Berger for questions and comments.
The research was supported by the European Union's Horizon 2020 research and innovation programme through the Marie Sklodowska-Curie grant agreement 101023017.

\end{document}